\documentclass[a4paper]{amsart}

\usepackage{a4wide}
\usepackage{amssymb,amsmath}
\usepackage{bbm}
\usepackage{txfonts}
\usepackage{mathrsfs}
\usepackage{enumerate}
\usepackage{nccmath}
\usepackage{graphicx}
\usepackage{mathptmx}

\usepackage{bm,bbold}
\usepackage{multirow}
\usepackage{subfig, hyphenat}
\usepackage{natbib}

\providecommand{\url}[1]{#1}
\numberwithin{equation}{section}

\usepackage[usenames,dvipsnames]{xcolor}
\usepackage[colorlinks,linkcolor=MidnightBlue,citecolor=OliveGreen]{hyperref}

\usepackage[capitalise]{cleveref}

\crefname{formula}{formula}{formul\ae}
\creflabelformat{enumi}{#2#1#3)}

\theoremstyle{plain}
\newtheorem{formula}{Formula}
\newtheorem{theorem}{Theorem}
\newtheorem{lemma}{Lemma}
\newtheorem{proposition}{Proposition}

\newtheorem{definition}{Definition}

\theoremstyle{remark}

\newcommand{\convd}{\xrightarrow{d}}
\newcommand{\EE}{\mathbb{E}}
\newcommand{\Pb}{\mathbb{P}}
\newcommand{\NN}{\mathbb{N}}

\newcommand{\R}{\mathbb{R} }

\newcommand{\bdeltab}{\delta}
\newcommand{\Deltab}{{\boldsymbol\Delta}}
\newcommand{\M}{\boldsymbol{M}}

\newcommand{\dfac}[1]{(#1)^!}

\DeclareMathOperator{\BJ}{J}
\DeclareMathOperator{\BY}{Y}

\begin{document}

\title{On the Markov transition kernels for first\hyp{}passage percolation on the ladder}
\author{Eckhard Schlemm}

\address{Wolfson College, Cambridge University}
\email{es555@cam.ac.uk}


\begin{abstract}
We consider the first\hyp{}passage percolation problem on the random graph with vertex set $\NN \times\{0,1\}$, edges joining vertices at Euclidean distance equal to unity and independent exponential edge weights. We provide a central limit theorem for the first\hyp{}passage times $l_n$ between the vertices $(0,0)$ and $(n,0)$, thus extending earlier results about the almost sure convergence of $l_n/n$ as $n\to\infty$. We use generating function techniques to compute the $n$-step transition kernels of a closely related Markov chain which can be used to calculate explicitly the asymptotic variance in the central limit theorem.
\end{abstract}

\maketitle

\section{Introduction}
The subject of first\hyp{}passage percolation, introduced in \citet{hammersley1965} in 1965, is the study of shortest paths in random graphs. Let $G=(V,E)$ be a graph with vertex set $V$ and unoriented edges $E\subset V^2$ and assume that there is a weight function $w:E\to\R_+$. For vertices $u,v\in V$, a path joining $u$ and $v$ in $G$ is a sequence of vertices $\boldsymbol{p}_{u\to v}=\{u=p_0,p_1,\ldots,p_{n-1},p_n=v\}$ such that $(p_\nu,p_{\nu+1})\in E$ for $0\leq\nu<n$. The weight $w(\boldsymbol{p}_{u\to v})$ of such a path is defined as the sum of the weights of the comprising edges, $w(\boldsymbol{p}_{u\to v})\coloneqq\sum_{\nu=0}^{n-1}{w((p_\nu,p_{\nu+1}))}$. The first\hyp{}passage time between the vertices $u$ and $v$ is denoted by $d_G(u,v)$ and defined as $d_G(u,v)\coloneqq	\inf{\{w(\boldsymbol{p}),\text{$\boldsymbol{p}$ a path joining $u$ and $v$ in $G$}\}}$. 

First-passage percolation can be considered a model for the spread of a fluid through a random porous medium; it differs from ordinary percolation theory in that it puts special emphasis on the dynamical aspect of how long it takes for certain points in the medium to be reached by the fluid. Important applications include the spread of infectious diseases (\cite{altmann1993}) and the analysis of electrical networks \citep{grimmet1984}. Recently, there has also been an increased interest in first\hyp{}passage percolation on graphs where not only the edge-weights, but the edge-structure itself is random. These models, including the Gilbert and Erd{\H{o}}s-R\'enyi random graph, were investigated in \citet{bhamidi2010, sood2005, hofstad2001} and found to be a useful approximation to the internet as well as telecommunication networks.

Usually, however, the underlying graph is taken to be $\mathbb{Z}^2$ and the edge weights are independent random variables with some common distribution $\mathbb{P}$, see e.\,g.\ \citet{smythe1978} and references therein. Interesting mathematical questions arising in this context include asymptotic properties of the sets $B_t\coloneqq\{u\in\mathbb{Z}^2:d_{\mathbb{Z}^2}(0,u)\leq t\}$ \citep{kesten1987,seppaelaiinen1998}, and the limiting behaviour of $d_{\mathbb{Z}^2}((0,0),(n,0))/n$. The latter is known to converge, under weak assumptions on $\mathbb{P}$, to a deterministic constant, called the {\it first\hyp{}passage percolation rate}; the computation of this constant has proved to be a very difficult problem and has not yet been accomplished even for the simplest choices of $\mathbb{P}$ \citep{graham1995}. An exception to this is the case when the underlying graph $G$ is essentially one-dimensional \citep{flaxman2006fpp,renlund2010,schlemm2009}.

In this paper we consider the first\hyp{}passage percolation problem on the ladder $G$, a particular essentially one-dimensional graph, for which the first\hyp{}passage percolation rate is known \citep{renlund2010,schlemm2009}.  We extend the existing results about the almost sure convergence of $d_{G}((0,0),(n,0))/n$ as $n\to\infty$ by providing a central limit theorem as well as giving a complete description of the $n$-step transition kernels of a closely related Markov chain. Our results can be used to explicitly compute the asymptotic variance in the central limit theorem. They are also the basis for the quantitative analysis of any other statistic related to first\hyp{}passage percolation in this model. In particular, knowledge of the higher order transition kernels is the starting point for the computation of the distribution of the rungs which are part of the shortest path. The ladder model is worth studying because it is one of the very few situations where a complete explicit description of the finite\hyp{}time behaviour of the first\hyp{}passage percolation times can be given.

The structure of the paper is as follows: In \cref{section-FPP-ladder} we describe the model and state our results; \cref{section-proofs}, which contains the proofs, is divided in three subsections. The first is devoted to the central limit theorem, the second presents some explicit evaluations of infinite sums which are needed in \cref{section-proofs2}, where the main theorem about the transition kernels is proven. We conclude the paper with a brief discussion.

We use the notation $\bdeltab_{p,q}$ for the Kronecker delta and $\Theta_{p,q}$ as well as $\tilde\Theta_{p,q}$ for versions of the discretized Heaviside step function, precisely
\begin{equation*}
\bdeltab_{p,q}\coloneqq\begin{cases}
                          1 & \text{ if } p=q\\
			  0 &	\text{else}
                         \end{cases},\quad \Theta_{p,q}\coloneqq\begin{cases}
							1 & \text{ if } p\geq q\\
							0 & \text{else}
							\end{cases},\quad \tilde\Theta_{p,q}\coloneqq\begin{cases}
											      1 & \text{ if } p\leq q\\
											      0 & \text{else}
											      \end{cases}.
\end{equation*}
The symbol $(k)^!$ stands for $k!(k+1)!$. We denote by $\R $ the real numbers and by $\mathbb{Z}$ the integers. A subscript $+$ ($-$) indicates restriction to the positive (negative) elements of a set. The symbol $\gamma$ stands for the Euler-Mascheroni constant and $\EE$ denotes expectation.
\section{First-passage percolation on the ladder}
\label{section-FPP-ladder}
\begin{figure}
\includegraphics[width=\textwidth]{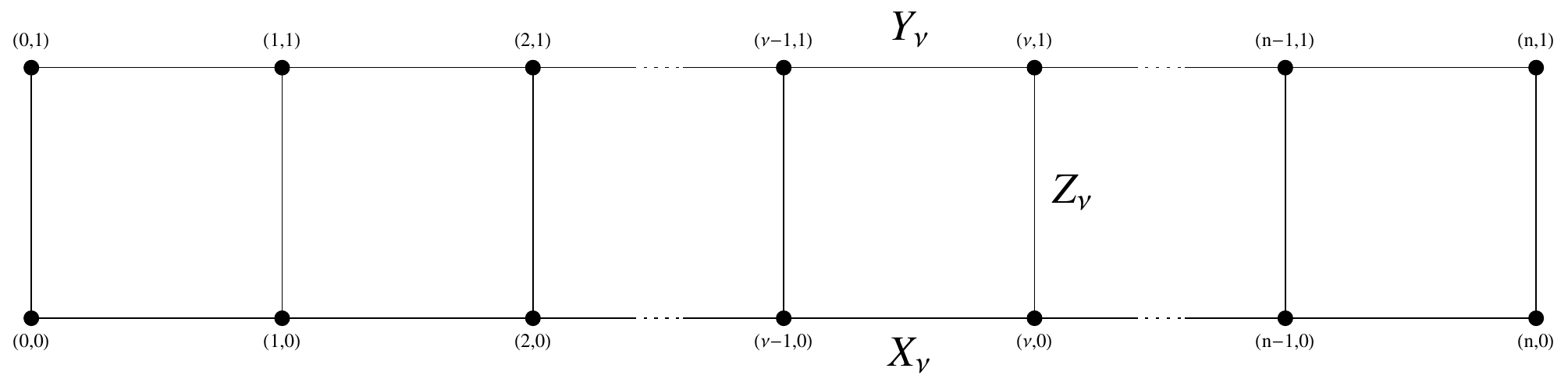}
\caption{The ladder graph $G_n$. The edge weights $X_\nu$, $Y_\nu$, $Z_\nu$ are independent exponential random variables.}
\label{fig-graph}
\end{figure}
In this paper we further investigate a first\hyp{}passage percolation model which has been considered before in \citet{renlund2010} and also in \citet{schlemm2009}. We denote by $G_n$ the graph with vertex set $V_n=\{0,\ldots,n\}\times\{0,1\}$ and edges
\begin{align*}
\mathscr{X}_n &= \{((i,0),(i+1,0)):0\leq i<n\},\\
\mathscr{Y}_n &= \{((i,1),(i+1,1)):0\leq i<n\},\\
\mathscr{Z}_n &= \{((i,0),(i,1)):0\leq i\leq n\}.
\end{align*}
The edge weights are independent exponentially distributed random variables which are labelled in the obvious way $X_i$, $Y_i$ and $Z_i$, see also \cref{fig-graph}. By time-scaling it is no restriction to assume that the edge weights have mean one. We further denote by $l_n$ the length of the shortest path from $(0,0)$ to $(n,0)$ in the graph $G_n$, by $l_n'$ the length of the shortest path from $(0,0)$ to $(n,1)$ and by $\Delta_n$ the difference between the two, $\Delta_n=l_n'-l_n$ It has been shown in \citet{schlemm2009} and also, by a different method, in \citet{renlund2010} that $\lim_{n\to\infty}{l_n/n}$ almost surely exists and is equal to the constant $\chi=\frac{3}{2}-\frac{\BJ_1(2)}{2\BJ_2(2)}$, where $\BJ_\nu$ are Bessel functions of the first kind. (See \cref{definition-Bessel} or \citet{abramowitz1992} for a comprehensive treatment.) This constant is called the first\hyp{}passage percolation rate for our model. The method employed in \citet{schlemm2009} to obtain this result built on \citet{flaxman2006fpp} and consisted in showing that there exists an ergodic $\R \times \R _+^3$\hyp{}valued Markov chain $\M=(M_n)_{n\geq 0}$ with stationary distribution $\tilde\pi$ and a function $f:\R \times \R _+^3\to\R $ such that
\begin{equation*}
\chi = \EE f(M_\infty) = \int_{\R \times\R _+^3}{f(m)\tilde\pi(\mathbbm{d}m)}.
\end{equation*}
Explicitly,
\begin{equation*}
M_n=(\Delta_n,X_{n+1},Y_{n+1},Z_{n+1})\quad \text{and} \quad f:({r},x,y,z)\mapsto\min\{{r}+y+z,x\}.
\end{equation*}
Throughout we write $m=({r},x,y,z)$ for some element of the state space $\R \times\R _+^3$. In order to better understand the first\hyp{}passage percolation problem on the ladder it is important to know the higher-order transition kernels $\tilde K^n:\left(\R\times\R_+^3\right)\times \left(\R\times\R_+^3\right) \to\R _+$ of the Markov chain $\M$. They completely determine the dynamics of the model and are defined as
\begin{equation}
\label{eq-DefKtilde}
\tilde K^n(m',m)\mathbbm{d}m = \Pb\left(M_n\in \mathbbm{d}m|M_0=m'\right),\quad m,m'\in\R \times \R _+^3.
\end{equation}
The first result shows that it is sufficient to control the transition kernels $K^n:\R \times \R \to\R _+$ of the Markov chain $\Deltab=(\Delta_n)_{n\geq 0}$ which are analogously defined as
\begin{equation}
\label{eq-DefK}
K^n({r'},{r})\mathbbm{d}{r} = \Pb\left(\Delta_n\in \mathbbm{d}{r}|\Delta_0={r'}\right),\quad {r},{r'}\in\R .
\end{equation}
For convenience we define $K^0({r'},{r})\coloneqq\bdeltab_{{r'}}({r})$, the Dirac distribution.
 \begin{proposition}
\label{lemma1}
For any $n\geq 1$, denote by $\tilde K^n$ the $n$-step transition kernel of $\M$ defined in \cref{eq-DefKtilde}. Then
\begin{equation}
\tilde K^n\left(m',m\right) = e^{-(x+y+z)}K^{n-1}\left(\min\{{r'}+y',x'+z'\}-\min\{{r'}+y'+z',x'\},{r}\right).
\end{equation}
Moreover, the stationary distribution $\tilde \pi$ of $\M$ is given by
\begin{equation}
\label{eq-statdistM}
\tilde\pi(\mathbbm{d}m) = e^{-(x+y+z)}\mathbbm{d}^3(x,y,z)\pi(\mathbbm{d}{r}),
\end{equation}
where
\begin{equation}
\label{eq-statdistDelta}
\pi(\mathbbm{d}{r}) = \frac{1}{2\BJ_2(2)}e^{-\frac{3}{2}|{r}|}\BJ_1\left(2e^{-\frac{1}{2}|{r}|}\right)\mathbbm{d}{r}
\end{equation}
is the stationary distribution of $\Deltab$.
\end{proposition}
Next, we state a central limit theorem for first\hyp{}passage percolation times on the ladder which has been implicit in \citet{schlemm2009} and which was the motivation for the current paper. In \citet{ahlberg2009, chatterjee2009arxiv} a central limit theorem has been obtained for first\hyp{}passage times on fairly general one-dimensional graphs by a different method. The question of how to compute the asymptotic variance has, however, not been addressed there. We denote by $\bar f$ the mean corrected function $f-\chi$.
\begin{theorem}
\label{theoremCLT}
For any integer $n\geq 0$, let $l_n$ denote the first\hyp{}passage time between $(0,0)$ and $(n,0)$ in the ladder graph $G_n$. Then there exists a positive constant $\sigma^2$ such that
\begin{equation}
\frac{l_n-n\chi}{\sqrt{n}}\convd N(0,\sigma^2),
\end{equation}
where $N(0,\sigma^2)$ is a normally distributed random variable with mean zero and variance $\sigma^2$ and $\convd$ denotes convergence in distribution. Moreover,
\begin{equation}
\label{eq-formulasigma}
\sigma^2 = \int_{\R \times\R _+^3}{\bar f(m)^2\tilde\pi(\mathbbm{d}m)} + 2 \sum_{n=1}^\infty{\int_{\R \times\R _+^3}{\bar f(m)P^n\bar f(m)\tilde\pi(\mathbbm{d}m})},
\end{equation}
where
\begin{equation}
P^n\bar f(m) = \EE\left[\bar f(M_n)|M_0=m\right]=\int_{\R \times\R _+^3}{\bar f(m')\tilde K^n(m,m')\mathbbm{d}m'}.
\end{equation}
\end{theorem}
\Cref{eq-formulasigma} shows that in order to evaluate the asymptotic variance of the first\hyp{}passage times one must know the transition kernels $\tilde K^n$. In the next theorem we therefore explicitly describe the structure of the transition kernel $K^n$ and thus, by \cref{lemma1}, of $\tilde K^n$. To state the formulas in a compact way we define the five functions
\begin{align*}
\operatorname{S}^1(z) =& \frac{z-2\BJ_2(2\sqrt{z})}{z},\qquad\qquad \operatorname{S}^2(z) = \frac{2(z-1)+2\BJ_0(2\sqrt{z})}{z},\\
\operatorname{G}(z) =& -\frac{z^{3}}{4}\left[5-4\gamma+2\pi\frac{\BY_2(2\sqrt{z})}{\BJ_2(2\sqrt{z})}-2\log z\right],\\
\alpha(z) =& \frac{z^4}{4\sqrt{z}\BJ_2(2\sqrt{z})\left[\sqrt{z}\BJ_0(2\sqrt{z})+(z-1)\BJ_1(2\sqrt{z})\right]}\quad\text{ and }\\
\operatorname{H}(z) =&\frac{z^2}{2(1-z)}\left[3z-3+2\pi\BY_0(2\sqrt{z})+\BJ_0(2\sqrt{z})\left(5-4\gamma-2\log z\right)\right].
\end{align*}
The Bessel functions $\BJ_\nu$ and $\BY_\nu$ are defined in \cref{definition-Bessel} and treated comprehensively in \citet{abramowitz1992}.
\begin{theorem}
\label{theorem1}
The transition kernels $K^n$, defined in \cref{eq-DefK}, satisfy $K^n({r'},{r})=K^n(-{r'},-{r})$. For ${r}\geq0$ the values $K^n({r'},{r})$ are given by
\begin{equation}
\label{eq-kernelMtoDelta}
\begin{cases}
                \sum_{p,q=0}^n {a_{p,q}^ne^{p{r'}-(q+2) {r}}} 			& {r'}\leq 0 	\\
		\frac{(-1)^{n-1}e^{{r'}-(n+1){r}}}{\dfac{n-1}}+\sum_{p=0}^{n-2}{\frac{(-1)^n{r'}e^{-p({r'}-{r})-n{r}}}{\dfac{p}\dfac{n-p-2}}}+\sum_{p,q=0}^n {b_{p,q}^n e^{-p{r'}-(q+2) {r}}} 	& 0<{r'}\leq {r} 	\\
		\frac{(-1)^{n-1}e^{-(n-1){r'}-{r}}}{\dfac{n-1}}+\sum_{p=0}^{n-2}{\frac{(-1)^n{r} e^{-p({r'}-{r})-n{r}}}{\dfac{p}\dfac{n-p-2}}}+\sum_{p,q=0}^n {c_{p,q}^n e^{-p{r'}-(q+2) {r}}}		& {r'}>{r}
                \end{cases},
\end{equation}
where the coefficients $a^n_{p,q}$, $b^n_{p,q}$ and $c^n_{p,q}$ are determined by their generating functions:
\begin{enumerate}[i)]
\item\label{theorem1-item1} the generating functions $\operatorname{A}_{p,q}(z)=\sum_{n=1}^\infty{a^n_{p,q}z^n}$, $p,q\geq 0$, are given by
\begin{subequations}
\label{eq-DefinitionA}
\begin{flalign}
\label{eq-DefinitionA1}\qquad\qquad&\operatorname{A}_{1,q}(z) = \frac{(-z)^q}{\dfac{q}}\alpha(z),&\\
\label{eq-DefinitionAp}\qquad\qquad&\operatorname{A}_{p,q}(z) =\frac{2(-z)^{p-1}}{\dfac{p}}\operatorname{A}_{1,q}(z),\quad p\geq 2,&\\
\label{eq-DefinitionA0}\qquad\qquad&\operatorname{A}_{0,q}(z) = \frac{\operatorname{S}^2(z)}{1-z}\operatorname{A}_{1,q}(z);&
\end{flalign}
\end{subequations}
\item\label{theorem1-item2} the generating functions $\operatorname{B}_{p,q}(z)=\sum_{n=1}^\infty{b^n_{p,q}z^n}$, $p,q\geq 0$, are given by  
\begin{subequations}
\label{eq-DefinitionB}
\begin{flalign}
\label{eq-DefinitionB1}\qquad\qquad&\operatorname{B}_{1,q}(z) =\frac{(-z)^q}{\dfac{q}}\operatorname{G}(z)-\operatorname{A}_{1,q}(z),&\\
\label{eq-DefinitionBp}\qquad\qquad&\operatorname{B}_{p,q}(z) =\frac{2(-z)^{p-1}}{\dfac{p}}\operatorname{B}_{1,q}(z)+\frac{(-z)^{p+q+2}}{\dfac{p}\dfac{q}}\sum_{k=2}^p{\frac{2k+1}{k(k+1)}},\quad p\geq 2,&\\
\label{eq-DefinitionB0}\qquad\qquad&\operatorname{B}_{0,q}(z) = \frac{\operatorname{S}^2(z)}{1-z}\operatorname{B}_{1,q}(z)+\frac{(-z)^q}{\dfac{q}}\operatorname{H}(z);&
\end{flalign}
\end{subequations}
\item\label{theorem1-item3} the generating functions $\operatorname{C}_{p,q}(z)=\sum_{n=1}^\infty{c^n_{p,q}z^n}$, $p,q\geq 0$, are given by
\begin{subequations}
\label{eq-DefinitionC}
\begin{flalign}
\label{eq-DefinitionC0}\qquad\qquad&\operatorname{C}_{0,q}(z) = d_qz^{q+2} + \operatorname{B}_{0,q}(z) - \frac{(-z)^{q+2}}{\dfac{q}},&\\
\label{eq-DefinitionC1}\qquad\qquad&\operatorname{C}_{1,q}(z) =\left[\operatorname{S}^1(z)+\frac{z\operatorname{S}^2(z)}{2(1-z)}\right]\operatorname{A}_{1,q}(z)-\frac{z}{2}\operatorname{C}_{0,q}(z),&\\
\label{eq-DefinitionCp}\qquad\qquad&\operatorname{C}_{p,q}(z) =\frac{2(-z)^{p-1}}{\dfac{p}}\operatorname{C}_{1,q}(z),\quad p\geq 2&
\end{flalign}
\end{subequations}
and the numbers $d_q$ are determined by their generating function $D(z)=\sum_{q=0}^\infty{d_q z^q}$ given by
\begin{equation}
D(z) = \frac{1}{z}\left[\sqrt{z}\BJ_1(2\sqrt{z})\left(2\gamma+\log z\right)-\pi\sqrt{z}\BY_1(2\sqrt{z})-1\right].
\end{equation}
\end{enumerate}
\end{theorem}
By the properties of generating functions the coefficients $a^n_{p,q}$, $b^n_{p,q}$ and $c^n_{p,q}$ are determined by the derivatives of $\operatorname{A}_{p,q}$, $\operatorname{B}_{p,q}$ and $\operatorname{C}_{p,q}$ evaluated at zero. These derivatives are routinely calculated to any order with the help of computer algebra systems such as Mathematica. \Cref{tab:coeff} exemplifies the theorem by reporting the values of the coefficients $a^n_{p,q}$, $b^n_{p,q}$, $c^n_{p,q}$ in the case $n=4$. Using these results the expression \labelcref{eq-formulasigma} for $\sigma^2$ can be evaluated explicitly in terms of certain integrals of hypergeometric functions; the computations, however, are quite involved and the final result rather lengthy, so we decided not to include them here.
\begin{table}
\centering
\subfloat[values of $a_{p,q}^4$]{\label{tab:coeffa}
  \begin{tabular}{ccccccc}
		   &&	&\multicolumn{4}{c}{$q$}\\
		   &&	&$0$	& $1$			& $2$			& $3$			\\[3pt]
\cline{4-7}\\[1pt]
\multirow{5}{*}{p} &\multicolumn{1}{c|}{$0$}&$\,$	&$\frac{115}{96}$	& $-\frac{17}{36}$	& $\frac{1}{24}$	& $0$		\\[3pt]
		   &\multicolumn{1}{c|}{$1$}&$\,$ 	&$\frac{11}{36}$	& $-\frac{1}{36}$	& $\frac{1}{12}$	& $-\frac{1}{144}$\\[3pt]
		   &\multicolumn{1}{c|}{$2$}&$\,$ 	&$-\frac{11}{108}$	& $\frac{1}{12}$	& $-\frac{1}{72}$	& $0$		\\[3pt]
		   &\multicolumn{1}{c|}{$3$}&$\,$ 	&$\frac{1}{72}$	& $\frac{1}{144}$	& $0$			& $0$			\\[3pt]
		   &\multicolumn{1}{c|}{$4$}&$\,$ 	&$-\frac{1}{1440}$	& $0$			& $0$			& $0$		
\end{tabular}
}
\hspace{13pt}\subfloat[values of $b_{p,q}^4$]{\label{tab:coeffb}
  \begin{tabular}{cccc}
		   \multicolumn{4}{c}{$q$}\\
		   $0$	& $1$			& $2$			& $3$		\\[3pt]
\cline{1-4}\\[1pt]
$\frac{721}{432}$	& $-\frac{13}{12}$	& $\frac{1}{12}$	& $0$		\\[3pt]
$-\frac{241}{540}$	& $\frac{17}{48}$	& $\frac{1}{36}$	& $0$		\\[3pt]
$\frac{3}{16}$		& $\frac{1}{36}$	& $0$			& $0$		\\[3pt]
$\frac{1}{216}$		& $0$			& $0$			& $0$		\\[3pt]
$0$			& $0$			& $0$			& $0$			
\end{tabular}
}
\hspace{13pt}\subfloat[values of $c_{p,q}^4$]{\label{tab:coeffc}
    \begin{tabular}{cccc}
		   \multicolumn{4}{c}{$q$}\\
		   $0$	& $1$			& $2$			& $3$		\\[3pt]
\cline{1-4}\\[1pt]
$\frac{721}{432}$	& $-\frac{13}{12}$	& $\frac{5}{18}$	& $0$		\\[3pt]
$-\frac{241}{540}$	& $\frac{17}{48}$	& $\frac{1}{36}$	& $0$		\\[3pt]
$-\frac{1}{144}$	& $\frac{1}{36}$	& $0$			& $0$		\\[3pt]
$\frac{1}{216}$		& $0$			& $0$			& $0$		\\[3pt]
$0$			& $0$			& $0$			& $0$			
\end{tabular}
}
\caption{Coefficients of the four-step transition kernel $K^4$, as given in \cref{eq-kernelMtoDelta}} 
\label{tab:coeff}
\end{table}

\section{Proofs}
\label{section-proofs}
\subsection{Proofs of Proposition \ref{lemma1} and Theorem \ref{theoremCLT}}
\label{section-proofs1}
In this section we present the proofs of the relation between the Markov chains $\M$ and $\Deltab$ and of the central limit theorem.

\begin{proof}[Proof of \cref{lemma1}]
Since
\begin{align*}
\Delta_n=&l_n'-l_n\\
	=&\min\{l_{n-1}'+Y_n,l_{n-1}+X_n+Z_n\} - \min\{l_{n-1}'+Y_n+Z_n,l_{n-1}+X_n\}\\
	=&\min\{\Delta_{n-1}+Y_{n-1},X_{n-1}+Z_{n-1}\}-\min\{\Delta_{n-1}+Y_{n+1}+Z_{n-1},X_{n-1}\}
\end{align*}
it follows at once that $M_0$ being equal to some $m'=({r'},x',y',z')$ implies $\Delta_1=\min\{{r'}+y',x'+z'\}-\min\{{r'}+y'+z',x'\}$ and thus the Markov property of $\Deltab$ together with the independence of the edge weights implies that for any integer $n>1$ the conditional probability $\Pb\left(M_n\in \mathbbm{d}m|M_0=m'\right)$ is given by
\begin{align*}
e^{-x+y+z}\mathbbm{d}^3(x,y,z)\Pb\left(\Delta_n\in\mathbbm{d}{r}|\Delta_1=\min\{{r'}+y',x'+z'\}-\min\{{r'}+y'+z',x'\}\right).
\end{align*}
The homogeneity of the Markov chain $\Deltab$ then implies \cref{eq-kernelMtoDelta} because for $n=1$ we clearly have
\begin{equation*}
\Pb\left(M_1\in \mathbbm{d}m|M_0=m'\right)=e^{-(x+y+z)}\bdeltab_{\min\{{r'}+y',x'+z'\}-\min\{{r'}+y'+z',x'\}}({r})\mathbbm{d}m.
\end{equation*}
\Cref{eq-statdistM} about the stationary distribution of $\M$ is a direct consequence of the fact that the edge weights $X_{n+1}$, $Y_{n+1}$ ,$Z_{n+1}$ are independent of $\Delta_n$ and expression \labelcref{eq-statdistDelta} was derived in \citet[Proposition 5.5.]{schlemm2009}.
\end{proof}
Next, we prove the central limit theorem and the formula for the asymptotic variance $\sigma^2$.
\begin{proof}[Proof of \cref{theoremCLT}]
We apply the general result \citet[Theorem 4.3.]{chen1999} for functionals of ergodic Markov chains on general state spaces. We first note that
\begin{equation*}
\int_{\R\times \R_+^3}f(m)^2\tilde\pi(\mathbb{d}m)=\frac{2\BJ_1(2)-3\BJ_0(2)+{}_2\operatorname{F}_3\left(\{1,1\},\{2,2,2\};-1\right)-1}{\BJ_2(2)}<\infty.
\end{equation*}
It then suffices to prove that the Markov chain $\M$ is uniformly ergodic, which is equivalent to showing that the Markov chain $\Deltab$ is uniformly ergodic. We use the drift criterion \citet[Theorem B]{aldous1997}, which asserts that if there is a sufficiently strong drift towards the centre of the state space of a Markov chain, it is uniformly ergodic. Using the Lyapunov function $V(r) = 1-e^{-|r|}$ as well as 
\begin{equation}
\label{eq-onestepkernel}
K^{}({r'},{r})=\begin{cases}
                           e^{-|{r}|} & \text{if }{r'}<{r}<0 \vee {r'}>{r}>0\\
			    e^{-|{r'}-2{r}|}	& \text{else}
                           \end{cases}
\end{equation}
for the one step transition kernel of $\Deltab$ \citep[Proposition 5.1]{schlemm2009} we obtain
\begin{align*}
\psi(r) \coloneqq P^1V(r) =& 1-\EE_r e^{-|\Delta_1|}\\
	=& 1-\int_\R{e^{-|\rho|}K^{}(r,\rho)d\rho}=\frac{1}{6}\left[3-2e^{-|r|}+e^{-2|r|}\right].
\end{align*}
It is easy to check that
\begin{equation*}
\psi(r)\leq V(r)-\frac{1}{10},\quad |r|\geq 1\quad \text{and}\quad \sup_{|r|\leq 1}\psi(r)\leq \sup_{r\in \R}\psi(r)=\frac{1}{2}<\infty.
\end{equation*}
Since the interval $[-1,1]$ is compact and has positive invariant measure it is a small set \citep[Remark 2.7.]{Nummelin1982} and it follows that $\Deltab$ is uniformly ergodic, which completes the proof.
\end{proof}

\subsection{Summation formulas}
\label{section-summation}

In this section we derive some summation formulas which we will use in the proofs in \cref{section-proofs2}. Some of them are well\hyp{}known, others can be checked with computer algebra systems such as Mathematica, a few (\cref{formula-U1,formula-Upsilon1,formula-U2,formula-Upsilon2}) seem to be new. The sums will be evaluated explicitly in terms of Bessel and generalized hypergeometric functions, which we now define.
\begin{definition}[Bessel function]
\label{definition-Bessel}
Let $\lambda$ be a real number in $\R \backslash\mathbb{Z}_-$. The {\it Bessel function of the first kind of order $\lambda$}, denoted by $\BJ_\lambda$, is defined by the series representation
\begin{equation}
 \label{eq-DefJ}
\BJ_\lambda(x) = \sum_{k=0}^\infty{\frac{(-1)^k}{\operatorname{\Gamma}(k)\operatorname{\Gamma}(k+\lambda+1)}\left(\frac{x}{2}\right)^{2k+\lambda}}.
\end{equation}
For any integer $\nu$ the {\it Bessel function of the second kind of order $\nu$}, denoted by $\BY_\nu$, is defined as
\begin{equation}
\label{eq-DefY}
\BY_\nu(x)=\lim_{\lambda\to\nu}{\frac{\BJ_\lambda(x)\cos{\lambda\pi}-\BJ_{-\lambda}(x)}{\sin{\lambda\pi}}}.
\end{equation}
\end{definition}
It is well\hyp{}known that Bessel functions satisfy the recurrence equation $\BJ_\nu(x)=\frac{2(\nu-1)}{x}\BJ_{\nu-1}(x)-\BJ_{\nu-2}(x)$ \citep[Formula 03.01.17.0002.01]{wolframfunctionsite}, which we use without further mentioning to simplify various expressions.
\begin{definition}[Generalized hypergeometric function]
For non\hyp{}negative integers $p\leq q$ and complex numbers $\boldsymbol{a}=a_1,\ldots,a_p$ and $\boldsymbol{b}=b_1,\ldots,b_q$, $b_j\notin\mathbb{Z}_-$, the {\it generalized hypergeometric function of order $(p,q)$ with coefficients $\boldsymbol{a}$, $\boldsymbol{b}$}, denoted by ${}_p\operatorname{F}_q\left(\boldsymbol{a},\boldsymbol{b};\cdot\right)$, is defined by the series representation
\begin{equation}
\label{eq-DefF}
{}_p\operatorname{F}_q\left(\boldsymbol{a},\boldsymbol{b};x\right)=\sum_{k=0}^\infty{\frac{(a_1)_k\cdots(a_p)_k}{(b_1)_k\cdots(b_q)_k}\frac{x^k}{k!}},
\end{equation}
where $(z)_k$ denotes the rising factorial defined by $(z)_k=\operatorname{\Gamma}(z+k)/\operatorname{\Gamma}(z)$.
\end{definition}
In particular we will encounter the regularized confluent hypergeometric functions ${}_0\tilde{\operatorname{F}}_1$, which are	 defined by ${}_0\tilde{\operatorname{F}}_1\left(\{\},\{b\};-z\right)=\frac{1}{\operatorname{\Gamma}(b)}{}_0\operatorname{F}_1\left(\{\},\{b\};-z\right)$; in the next lemma we relate their derivative with respect to $b$ to certain sums involving the harmonic numbers $\operatorname{H}_k\coloneqq\sum_{n=1}^k{1/n}$.
\begin{lemma}
\label{lemma-HG0F1}
Denote by ${}_0\tilde{\operatorname{F}}_1$ the regularized version of the hypergeometric function ${}_0\operatorname{F}_1$. It then holds that
\begin{enumerate}[i)]
 \item\label{lemma-HG0F1-item1} for every positive integer $\nu$,
\begin{subequations}
\label{eq-HG0F1}
\begin{equation}
\label{eq-HG0F1-1}
\left.\frac{\mathbbm{d}}{\mathbbm{d}b}{}_0\tilde{\operatorname{F}}_1\left(\{\},\{b\};-z\right)\right|_{b=\nu} = \gamma z^{\frac{1-\nu}{2}}\BJ_{\nu-1}\left(2\sqrt{z}\right)-\sum_{k=0}^\infty{\frac{(-z)^k}{k!(k+\nu-1)!}\operatorname{H}_{k+\nu-1}}.
\end{equation}
\item\label{lemma-HG0F1-item2} for every positive integer $\nu$,
\begin{align}
\left.\frac{\mathbbm{d}}{\mathbbm{d}b}{}_0\tilde{\operatorname{F}}_1\left(\{\},\{b\};-z\right)\right|_{b=-\nu}=&(-1)^{\nu-1}\gamma z^{\frac{\nu+1}{2}}\BJ_{\nu+1}\left(2\sqrt{z}\right)-\sum_{k=0}^\infty{\frac{(-z)^{k+\nu+1}}{k!(k+\nu+1)!}\operatorname{H}_k}\notag\\
\label{eq-HG0F1-2}  &+(-1)^\nu\sum_{k=0}^\nu{\frac{z^k(\nu-k)!}{k!}}.
\end{align}
\end{subequations}
\end{enumerate}
\end{lemma}
\begin{proof}
For part \labelcref{lemma-HG0F1-item1}, we differentiate the series representation \labelcref{eq-DefF} term by term. Using the definition of the Digamma function $\digamma$ as the logarithmic derivative of the Gamma function $\operatorname{\Gamma}$ as well as the relation $\digamma(k)=-\gamma+\operatorname{H}_{k-1}$, $k\in\mathbb{Z}_+$ \citep[Formula 06.14.27.0003.01]{wolframfunctionsite}, we get
\begin{equation*}
\left.\frac{\mathbbm{d}}{\mathbbm{d}b}\frac{1}{\operatorname{\Gamma}(k+b)}\right|_{b=\nu}=-\frac{\digamma(k+\nu)}{\operatorname{\Gamma}(k+\nu)}=\frac{\gamma-\operatorname{H}_{k+\nu-1}}{(k+\nu-1)!}
\end{equation*}
and thus
\begin{equation*}
\left.\frac{\mathbbm{d}}{\mathbbm{d}b}{}_0\tilde{\operatorname{F}}_1\left(\{\},\{b\};-z\right)\right|_{b=\nu} = \gamma\sum_{k=0}^\infty{\frac{(-z)^k}{k!(k+\nu-1)!}}-\sum_{k=0}^\infty{\frac{(-z)^k}{k!(k+\nu-1)!}\operatorname{H}_{k+\nu-1}} 
\end{equation*}
which concludes the proof of the first part of the lemma. Part \labelcref{lemma-HG0F1-item2} is shown in a completely analogous way, using the relation $\lim_{\mu\to -m}\digamma(\mu)/\operatorname{\Gamma}(\mu)=(-1)^{m+1}m!$ for every non\hyp{}negative integer $m$, which follows from the fact \citet[Formula 06.05.04.0004.01]{wolframfunctionsite} that the Gamma function has a simple pole at $-m$ with residue $(-1)^m/m!$.
\end{proof}

\begin{formula}
\label{formula-S1}
\begin{equation*}
\sum_{k=1}^\infty{\frac{(-z)^k}{k!(k+2)!}}=\frac{2\BJ_2(2\sqrt{z})-z}{2z}=-\frac{1}{2}\operatorname{S}^1(z).
\end{equation*}
\end{formula}
\begin{proof}
 Immediate from the definition \labelcref{eq-DefJ} of Bessel functions.
\end{proof}

\begin{formula}
\label{formula-S2}
\begin{equation*}
\sum_{k=1}^\infty{\frac{(-z)^k}{(k+1)!^2}}= \frac{1-z-\BJ_0(2\sqrt{z})}{z} = -\frac{1}{2}\operatorname{S}^2(z).
\end{equation*}
\end{formula}
\begin{proof}
This is also clear from the definition \labelcref{eq-DefJ} of Bessel functions.
\end{proof}

\begin{formula}
\label{formula-S3}
\begin{equation*}
\operatorname{S}^3(z)=\sum _{n=q+2}^{\infty } \frac{(-z)^{n-q}}{\dfac{n-q-2} (n-q)^2}=1-\BJ_2\left(2\sqrt{z}\right)-\sqrt{z} \BJ_1\left(2 \sqrt{z}\right).
\end{equation*}
\end{formula}
\begin{proof}
Shifting the index of summation $n$ by $q+2$ we obtain
\begin{align*}
\operatorname{S}^3(z)=&\sum _{n=0}^{\infty } \frac{(-z)^{n+2}}{n! (n+1)! (n+2)^2}=\sum _{n=0}^{\infty }(-z)^{n+2}\left[\frac{1}{(n+1)!(n+2)!}-\frac{1}{(n+2)!^2}\right]\\
      =& \sum _{n=0}^{\infty }\frac{(-z)^{n+1}}{\dfac{n}}+z-\sum _{n=0}^{\infty }\frac{(-z)^n}{n!^2}+1-z\\
      =&1-\BJ_2\left(2\sqrt{z}\right)-\sqrt{z} \BJ_1\left(2 \sqrt{z}\right)
\end{align*}
by the definition \labelcref{eq-DefJ} of Bessel functions.
\end{proof}

\begin{formula}
\label{formula-Sigma1}
\begin{equation*}
\operatorname{\Sigma}^1(\zeta z)=\sum_{q=0}^\infty\sum _{k=0}^{q-2} \frac{(-\zeta z)^q}{\dfac{k} \dfac{q-k-2} (k+2)^2}=\frac{\BJ_1\left(2\sqrt{\zeta z}\right)}{\sqrt{\zeta z}}\operatorname{S}^3(\zeta z).
\end{equation*}
\end{formula}
\begin{proof}
By Fubini's theorem we can interchange the order of summation and then shift the summation index $q$ by $k+2$ to get
\begin{align*}
\operatorname{\Sigma}^1(\zeta z)=&\sum _{k=0}^\infty\frac{(-\zeta z)^{k+2}}{\dfac{k}(k+2)^2}\sum_{q=0}^\infty \frac{(-\zeta z)^q}{\dfac{q}}.
\end{align*}
By definition \labelcref{eq-DefJ}, the second sum equals $\BJ_1\left(2\sqrt{\zeta z}\right)/\sqrt{\zeta z}$ and so the claim follows with \cref{formula-S3}.
\end{proof}

\begin{formula}
\label{formula-Sigma2}
\begin{equation*}
\operatorname{\Sigma}^2(\zeta z)=\sum_{q=0}^\infty{\sum _{k=1}^{q-2} \frac{2(-\zeta z)^q}{k! (k+2)! \dfac{q-k-2}}}=\left[2\BJ_2\left(2\sqrt{\zeta z}\right)-\zeta z\right]\sqrt{\zeta z}\BJ_1\left(2\sqrt{\zeta z}\right).
\end{equation*}
\end{formula}
\begin{proof}
We can interchange the order of summation and shift the index $q$ by $k+2$ to obtain
\begin{equation*}
\operatorname{\Sigma}^2(\zeta z)=\sum _{k=1}^\infty\frac{2(-\zeta z)^{k+2}}{k!(k+2)!}\sum_{q=0}^\infty \frac{(-\zeta z)^q}{\dfac{q}}
\end{equation*}
The first factor equals $2\zeta z\BJ_2\left(2\sqrt{\zeta z}\right)-(\zeta z)^2$ and the second factor equals $\BJ_1\left(2\sqrt{\zeta z}\right)/\sqrt{\zeta z}$, both by definition \labelcref{eq-DefJ}, and so the claim follows.
\end{proof}

\begin{formula}
\label{formula-T1}
\begin{equation*}
\operatorname{T}^1(\zeta z)=\sum_{q=0}^\infty{\frac{(-\zeta z)^q}{\dfac{q}}}=\frac{\BJ_1\left(2\sqrt{\zeta z}\right)}{\sqrt{\zeta z}}.
\end{equation*}
\end{formula}
\begin{proof}
Clear from the definition \labelcref{eq-DefJ}.
\end{proof}

\begin{formula}
\label{formula-T2}
\begin{equation*}
\operatorname{T}^2(\zeta z) =\sum_{q=0}^\infty{\frac{(-\zeta z)^{q+1}q}{(q+1)!^2}}=1-\BJ_0\left(2\sqrt{\zeta z}\right)-\sqrt{\zeta z}\BJ_1\left(2\sqrt{\zeta z}\right).
\end{equation*}
\end{formula}
\begin{proof}
This follows from the decomposition $\frac{q}{(q+1)!^2}=\frac{1}{\dfac{q}}-\frac{1}{(q+1)!^2}$ and definition \labelcref{eq-DefJ}.
\end{proof}

\begin{formula}
\label{formula-U1}
\begin{align*}
\operatorname{U}^1(z)=&\sum _{k=1}^{\infty } \frac{(-z)^{k+2} }{ k! (k+2)!}\sum _{l=2}^k \frac{2 l+1}{l (l+1)}\\
	 =&\frac{3z^2}{4} - z - 2 -  \pi  z \BY_2\left(2 \sqrt{z}\right)+\frac{\sqrt{z}}{2} \BJ_1\left(2 \sqrt{z}\right) \left[4 \gamma -3 +2\log z\right]\\
  &+\BJ_0\left(2\sqrt{z}\right)\left[1+\frac{5z}{2}-2 \gamma  z-z \log z\right]. 
\end{align*}
\end{formula}
\begin{proof}
First we note that
\begin{equation}
\label{eq-formulaU1-1}
\sum_{l=2}^k{\frac{2l+1}{l(l+1)}}=\sum_{l=2}^k\left[\frac{1}{l}+\frac{1}{l+1}\right]=2\operatorname{H}_k-\frac{5k+3}{2(k+1)},
\end{equation}
where $\operatorname{H}_k$ denotes the $k^{\text{th}}$ harmonic number. The first contribution to $\operatorname{U}^1(z)$ can therefore be computed as
\begin{align*}
\label{eq-formulaU1-2}
\frac{1}{2}\sum_{k=1}^\infty{\frac{(5k+3)(-z)^k}{k!(k+2)!(k+1)}}=&-\frac{1}{z}-\frac{3}{4}-\frac{1}{z}\sum_{k=0}^\infty{\frac{(-z)^k}{\dfac{k}}}+\frac{5}{2}\sum_{k=0}^\infty{\frac{(-z)^k}{k!(k+2)!}}\notag\\
  =&\frac{1}{4z^{3/2}}\left[4\BJ_1\left(2\sqrt{z}\right)+10\sqrt{z}\BJ_2\left(2\sqrt{z}\right)-3z^{3/2}-4\sqrt{z}\right].
\end{align*}
The other contribution to $\operatorname{U}^1(z)$ is, with the help of \cref{lemma-HG0F1},\labelcref{lemma-HG0F1-item2}, obtained as
\begin{equation*}
2\sum_{k=1}^\infty{\frac{(-z)^k}{k!(k+2)!}\operatorname{H}_k}=\frac{2}{z^2}\left[\gamma z\BJ_2\left(2\sqrt{z}\right)-z-1-\left.\frac{\mathbbm{d}}{\mathbbm{d}b}{}_0\tilde{\operatorname{F}}_1\left(\{\},\{b\};-z\right)\right|_{b=-1}\right].
\end{equation*}
By \citet[Formula 07.18.20.0015.01]{wolframfunctionsite},
\begin{equation*}
\left.\frac{\mathbbm{d}}{\mathbbm{d}b}{}_0\tilde{\operatorname{F}}_1\left(\{\},\{b\};-z\right)\right|_{b=-1}=\frac{1}{2}\left[\pi z\BY_2\left(2\sqrt{z}\right)-\sqrt{z}\BJ_1\left(2\sqrt{z}\right)\left[2+\log z\right]+\BJ_0\left(2\sqrt{z}\right)\left[z\log z-1\right]\right]
\end{equation*}
and the result follows after combining the last four displayed equations.
\end{proof}

\begin{formula}
\label{formula-Upsilon1}
\begin{equation*}
\operatorname{\Upsilon}^1(\zeta z)=\sum _{q=0}^{\infty }\sum _{k=1}^{q-2} \frac{(-\zeta z)^q }{\dfac{q-k-2}k!
(k+2)!}\sum _{l=2}^k \frac{2 l+1}{l (l+1)}=\frac{\BJ_1\left(2\sqrt{\zeta z}\right)}{\sqrt{\zeta z}}\operatorname{U}^1(\zeta z).
\end{equation*}
\end{formula}
\begin{proof}
Interchanging the order of the first two summations and shifting the index $q$ by $k+2$  we find that
\begin{equation*}
\operatorname{\Upsilon}^1(\zeta z)=\sum _{k=1}^\infty\frac{(-\zeta z)^{k+2}}{k!(k+2)!} \sum _{q=0}^\infty\frac{(-\zeta z)^q }{\dfac{q}}\sum _{l=2}^k \frac{2 l+1}{l (l+1)}.
\end{equation*}
As before the sum in the middle equals $\BJ_1\left(2\sqrt{\zeta z}\right)/\sqrt{\zeta z}$ and so the claim follows with \cref{formula-U1}.
\end{proof}

\begin{formula}
\label{formula-S4}
\begin{equation*}
\operatorname{S}^4(z)=\sum _{n=q+2}^{\infty } \frac{(-z)^{n-q}}{\dfac{n-q-2} (n-q-1)^2}=z^2{}_2\operatorname{F}_3\left(\{1,1\},\{2,2,2\};-z\right).
\end{equation*}
\end{formula}
\begin{proof}
After shifting the index $n$ by $q+2$ this follows immediately from the definition \labelcref{eq-DefF} of the hypergeometric function.
\end{proof}

\begin{formula}
\label{formula-U2}
\begin{align*}
\operatorname{U}^2(z)=&\sum _{k=1}^{\infty } \frac{(-z)^{k+2} }{(k+1)!^2}\sum _{l=2}^k \frac{2 l+1}{l (l+1)}\\
      =&\frac{z}{2}\left[-5+3 z+2 \pi  \BY_0\left(2\sqrt{z}\right)-2 z {}_2\operatorname{F}_3\left(\{1,1\},\{2,2,2\},-z\right)+\BJ_0\left(2 \sqrt{z}\right) (5-4 \gamma-2 \log z)\right].
\end{align*}
\end{formula}
\begin{proof}
The proof is analogous to \cref{formula-U1}. Using \cref{eq-formulaU1-1} the first contribution to $\operatorname{U}^2(z)$ is
\begin{align}
\label{eq-formulaU2-1}
\frac{1}{2}\sum _{k=1}^{\infty } \frac{(5k+3)(-z)^{k+2} }{(k+1)!^2(k+1)}=&\frac{5-3z}{2z}-\frac{5}{2z}\sum_{k=0}^\infty\frac{(-z)^k}{k!^2}-\sum_{k=0}^\infty\frac{(-z)^k}{(k+1)!^2(k+1)^2}\notag\\
    =&\frac{1}{2z}\left[5-3z-5\BJ_0\left(2\sqrt{z}\right)-2z{}_2\operatorname{F}_3\left(\{1,1\},\{2,2,2\};-z\right)\right],
\end{align}
where we used the definitions \labelcref{eq-DefJ,eq-DefF}. For the remaining part we first use \cref{lemma-HG0F1},\labelcref{lemma-HG0F1-item1} and the identity $\BJ_1\left(2\sqrt{z}\right)/\sqrt{z}-\BJ_2\left(2\sqrt{z}\right)=\BJ_0\left(2\sqrt{z}\right)$ to compute
\begin{align*}
&\sum_{k=1}^\infty{\frac{(-z)^k}{(k+1)!^2}\operatorname{H}_{k+2}}\\
=&\sum_{k=1}^\infty{\frac{(-z)^k}{(k+1)!(k+2)!}\operatorname{H}_{k+2}}+\sum_{k=1}^\infty{\frac{(-z)^k}{k!(k+2)!}\operatorname{H}_{k+2}}\\
  =&-\frac{3}{2}+\frac{1}{z}-\frac{1}{z}\left[\sum_{k=0}^\infty{\frac{(-z)^k}{\dfac{k}}\operatorname{H}_{k+1}}-z\sum_{k=0}^\infty{\frac{(-z)^k}{k!(k+2)!}\operatorname{H}_{k+2}}\right]\\
  =&-\frac{3}{2}+\frac{1}{z}-\frac{\gamma}{z}\BJ_0\left(2\sqrt{z}\right)+\frac{1}{z}\left[\left.\frac{\mathbbm{d}}{\mathbbm{d}b}{}_0\tilde{\operatorname{F}}_1\left(\{\},\{b\};-z\right)\right|_{b=2}-z\left.\frac{\mathbbm{d}}{\mathbbm{d}b}{}_0\tilde{\operatorname{F}}_1\left(\{\},\{b\};-z\right)\right|_{b=3}\right].
\end{align*}
It then follows from the explicit characterization of $\left.\frac{\mathbbm{d}}{\mathbbm{d}b}{}_0\tilde{\operatorname{F}}_1\left(\{\},\{b\};-z\right)\right|_{b=-\nu}$, $\nu\in\mathbb{Z}_+$, given in \citet[Formula 07.18.20.0013.01]{wolframfunctionsite} that
\begin{align*}
\sum_{k=1}^\infty{\frac{(-z)^k}{(k+1)!^2}\operatorname{H}_{k+2}}=&\frac{1}{2z^{3/2}}\left[\sqrt{z}\left[2-3z+\pi\BY_0\left(2\sqrt{z}\right)\right]-2\BJ_1\left(2\sqrt{z}\right)-\sqrt{z}\BJ_0\left(2\sqrt{z}\right)\left[2\gamma+\log z\right]\right].
\end{align*}
Using this we get
\begin{align}
\label{eq-formulaU2-2}
2\sum_{k=1}^\infty{\frac{(-z)^k}{(k+1)!^2}\operatorname{H}_k}=&2\sum_{k=1}^\infty{\frac{(-z)^k}{(k+1)!^2}\operatorname{H}_{k+2}}-2\sum_{k=1}^\infty{\frac{(-z)^k}{(k+1)!^2(k+2)}}-2\sum_{k=1}^\infty{\frac{(-z)^k}{(k+1)!^2(k+1)}}\notag\\
						=&\frac{1}{2z^{3/2}}\left[\sqrt{z}\left[2-3z+\pi\BY_0\left(2\sqrt{z}\right)\right]-2\BJ_1\left(2\sqrt{z}\right)-\sqrt{z}\BJ_0\left(2\sqrt{z}\right)\left[2\gamma+\log z\right]\right]\notag\\
						  &+\left[\frac{1}{2}-\frac{1}{z}+\frac{\BJ_1\left(2\sqrt{z}\right)}{z^{3/2}}\right]+1-{}_2\operatorname{F}_3\left(\{1,1\},\{2,2,2\};-z\right).
\end{align}
Combining \cref{eq-formulaU2-1,eq-formulaU2-2} completes the proof.
\end{proof}

\begin{formula}
\label{formula-Upsilon2}
\begin{equation*}
\operatorname{\Upsilon}^2(\zeta z) =\sum _{q=0}^{\infty} \sum _{k=1}^{q-1} \frac{(-\zeta z)^q }{\dfac{q-k-1} (k+1)!^2}\sum _{l=2}^k \frac{2 l+1}{l (l+1)}=-\frac{\BJ_1\left(2\sqrt{\zeta z}\right)}{(\zeta z)^{3/2}}\operatorname{U}^2(\zeta z).
\end{equation*}
\end{formula}
\begin{proof}
Interchanging the order of the first two summations and shifting the index $q$ by $k+1$ we find that
\begin{equation*}
\operatorname{\Upsilon}^2(\zeta z) = \sum _{k=1}^\infty\frac{(-\zeta z)^{k+2}}{(k+1)!^2} \sum _{q=0}^{\infty}\frac{(-\zeta z)^{q-1} }{\dfac{q}}\sum _{l=2}^k \frac{2 l+1}{l (l+1)}.
\end{equation*}
By definition \labelcref{eq-DefJ}, the middle sum is equal to $-\BJ_1\left(2\sqrt{\zeta z}\right)/(\zeta z)^{3/2}$ and so the result follows readily from \cref{formula-U2}.
\end{proof}

\begin{formula}
\label{formula-Sigma3}
\begin{equation*}
\operatorname{\Sigma}^3(\zeta z) =\sum_{q=0}^\infty{\sum _{k=1}^{q-1} \frac{2 (-\zeta z)^q}{\dfac{q-k-1} (k+1)!^2}}=\frac{2 \left[\zeta z -1 +\BJ_0\left(2 \sqrt{\zeta z }\right)\right] \BJ_1\left(2 \sqrt{\zeta z }\right)}{\sqrt{\zeta z}}.
\end{equation*}
\end{formula}
\begin{proof}
The same as before so we omit it.
\end{proof}

\begin{formula}
\label{formula-Sigma4}
\begin{align*}
\operatorname{\Sigma}^4(\zeta z) =&\sum _{q=0}^{\infty } \sum _{k=0}^{q-1} \frac{(-\zeta z)^q}{\dfac{k} \dfac{q-k-1} (k+1)^2}=-\sqrt{\zeta z } \BJ_1\left(2\sqrt{\zeta z }\right) {}_2\operatorname{F}_3\left(\{1,1\},\{2,2,2\};-\zeta z \right).
\end{align*}
\end{formula}
\begin{proof}
Using Fubini's theorem we get
\begin{equation*}
\operatorname{\Sigma}^4(\zeta z) =\sum _{k=0}^\infty\frac{(-\zeta z)_k}{\dfac{k}(k+1)^2} \sum _{q=k+1}^{\infty } \frac{(-\zeta z)^{q-k}}{ \dfac{q-k-1}}.
\end{equation*}
The first factor is equal to ${}_2\operatorname{F}_3\left(\{1,1\},\{2,2,2\};-\zeta z\right)$ by definition \labelcref{eq-DefF} and the second factor equals $-\sqrt{\zeta z}\BJ_1\left(2\sqrt{\zeta z}\right)$ by definition \labelcref{eq-DefJ}.
\end{proof}

Having evaluated these sums we now turn to relations between them which we will also need and which are proved by writing out the expressions and straightforward computations. The first one only involves the five functions appearing in the statement of \cref{theorem1}.
\begin{lemma}
\label{lemma-relation1}
For almost every complex number $z$ with respect to the Lebesgue measure on the complex plane,
\begin{equation}
\left[\frac{\operatorname{S}^2(z)}{2(1-z)}+\frac{1}{z}\right]\operatorname{G}(z)+\left[\frac{\operatorname{S}^1(z)}{z}+\frac{\operatorname{S}^2(z)}{1-z}+\frac{1}{z}\right]\alpha(z)+\frac{\operatorname{H}(z)}{2}+\frac{3z^2}{4}=0.
\end{equation}
\end{lemma}

\begin{lemma}
\label{lemma-relation2}
For almost every complex number $z$ with respect to the Lebesgue measure on the complex plane,
\begin{equation}
\operatorname{U}^1(z)+\operatorname{S}^3(z)+\left[\frac{\operatorname{S}^1(z)-1}{z}\right]\operatorname{G}(z)=\frac{3}{4}z^2-z-1.
\end{equation}
\end{lemma}

\subsection{Proof of Theorem \ref{theorem1}}
\label{section-proofs2}
In this section we prove the main result, \cref{theorem1}. First, however, we spend a closer look on the coefficients $a^n_{p,q}$, $b^n_{p,q}$ and $c^n_{p,q}$ which are defined implicitly through the generating functions \labelcref{eq-DefinitionA,eq-DefinitionB,eq-DefinitionC}.
\begin{lemma}
\label{initialconditions}
For any integers $p,q\geq 0$ the following hold:
\begin{align}
a^1_{p,q} = \bdeltab_{p,1}\bdeltab_{q,0},\quad b^1_{p,q}=0,\quad c^1_{p,q}=0.
\end{align}
\end{lemma}
\begin{proof}
This follows from evaluating the derivatives of the generating functions $\operatorname{A}_{p,q}$, $\operatorname{B}_{p,q}$, $\operatorname{C}_{p,q}$ at zero.
\end{proof}
Next we derive some useful relations between the coefficients $a^n_{p,q}$, $b^n_{p,q}$ and $c^n_{p,q}$. These will be the main ingredient in our inductive proof of \cref{theorem1}. The general strategy in proving the equality of two sequences $(s_n)_{n\geq 1}$ and $(\tilde s_n)_{n\geq 1}$ will be to compute their generating functions $\sum_{n\geq1}{s_nz^n}$ and $\sum_{n\geq1}{\tilde s_nz^n}$ and to show that they coincide for every $z$. The validity of this approach follows from the well\hyp{}known bijection between sequences of real numbers and generating functions (see e.\,g.\ \citet{wilf2006} for an introductory treatment.) We will constantly be making use of the convolution property of generating functions. By this we mean the simple fact that if $(s_n)_{n\geq 1}$ is a real sequence with generating function $S(z)$ and $(t_n)_{n\geq 1}$ is another such sequence with generating function $T(z)$, then the sequence of partial sums $(\sum_{\nu=1}^{n-1}{t_{n-\nu}s_\nu})_{n\geq 1}$ has generating function $S(z)T(z)$. We also encounter generating functions of sequences indexed by $q$ instead of $n$. In this case we denote the formal variable by $\zeta$ instead of $z$ and sums are understood to be indexed from zero to infinity.

\begin{lemma}
\label{lemmarecursionBeta}
For every integer $n\geq 1$, $q\geq 0$ the coefficients defined by the generating functions given in \cref{theorem1} satisfy the relation
\begin{equation}
\label{eq-lemmarecursionBeta}
\sum_{k=0}^{q-2}{\frac{b^n_{k,q-k-2}}{k+2}}-\sum_{k=0}^{q-2}{\frac{c^n_{k,q-k-2}}{k+2}}=\bdeltab_{q,n}\left[\frac{(-1)^n(n-1)}{n!^2}-\sum_{k=0}^{n-2}{\frac{(-1)^n}{\dfac{k}\dfac{n-k-2}(k+2)^2}}\right].
\end{equation}
\end{lemma}
\begin{proof}
These equations are true for all $n\geq 1$ if and only if the corresponding generating functions coincide. Multiplying both sides by $z^n$, summing over $n$ and using the recursive definitions of the generating functions as well as the fact that, by \cref{formula-S1,formula-S2},
\begin{equation*}
\sum_{k=1}^\infty{\frac{(-z)^k}{k!(k+2)!}}=-\frac{1}{2}\operatorname{S}^1(z),\qquad \sum_{k=1}^\infty{\frac{(-z)^k}{(k+1)!^2}}=-\frac{1}{2}\operatorname{S}^2(z),
\end{equation*}
we find that the claim of the lemma is equivalent to
\begin{align*}
0=&\sum_{k=1}^{q-2}{\frac{(-z)^q}{k!(k+2)!\dfac{q-k-2}}\sum_{l=1}^k{\frac{2l+1}{l(l+1)}}}+\frac{1}{2}\Theta_{q,2}\left[\frac{(-z)^q}{\dfac{q-2}}-z^q d_{q-2}\right]\\
  &+\frac{z^q}{2}\sum_{k=1}^{q-2}\frac{2(-1)^{k-1}}{k!(k+2)!}d_{q-k-2}-\Theta_{q,1}\left[\frac{(-z)^q(q-1)}{q!^2}-\sum_{k=0}^{q-2}{\frac{(-z)^q}{\dfac{k}\dfac{q-k-2}(k+2)^2}}\right]\\
  &+\frac{5}{4}z^3\sum _{k=1}^{q-2} \frac{2(-z)^{q-3}}{k! (k+2)! \dfac{q-k-2}},
\end{align*}
where we have used \cref{lemma-relation1} to simplify the coefficient of the last sum. To show this equality for all non\hyp{}negative integers $q$ we compare the $q$-generating functions and must then show that
\begin{align}
\label{eq-recursionBetaqGF}
0=&\operatorname{\Upsilon}^1(\zeta z)+\frac{(\zeta z)^2}{2}\left[\operatorname{T}^1(\zeta z)+\left((\operatorname{S}^1(\zeta z)-1\right)D(\zeta z)\right]-\operatorname{T}^2(\zeta z)+\operatorname{\Sigma}^1(\zeta z)+\frac{5}{4}\operatorname{\Sigma}^2(\zeta z),
\end{align}
where closed form expressions for
\begin{align*}
\operatorname{\Sigma}^1(\zeta z)  \coloneqq&\sum_{q=0}^\infty{\sum _{k=0}^{q-2} \frac{(-\zeta z)^q}{\dfac{k} \dfac{q-k-2} (k+2)^2}},\\ \operatorname{\Sigma}^2(\zeta z)  \coloneqq&\sum_{q=0}^\infty{\sum _{k=1}^{q-2} \frac{2(-\zeta z)^q}{k! (k+2)! \dfac{q-k-2}}},\\
\operatorname{T}^1(\zeta z) 	   \coloneqq&\sum_{q=0}^\infty{\frac{(-\zeta z)^q}{\dfac{q}}},\\ 
\operatorname{T}^2(\zeta z)       \coloneqq&\sum_{q=0}^\infty{\frac{(-\zeta z)^{q+1}q}{(q+1)!^2}}\qquad\text{and}\\
\operatorname{\Upsilon}^1(\zeta z)\coloneqq&\sum _{q=0}^{\infty }\sum _{k=1}^{q-2} \frac{(-\zeta z)^q }{\dfac{q-k-2}k!
(k+2)!}\sum _{l=2}^k \frac{2 l+1}{l (l+1)}&
\end{align*}
are derived in \cref{formula-Sigma1,formula-Sigma2,formula-T1,formula-T2,formula-Upsilon1}. Using these closed form formulas, \cref{eq-recursionBetaqGF} is seen to be identically true by simple algebra.
\end{proof}

\begin{lemma}
\label{lemmarecursionA}
For every integers $p,q\geq 0$ the sequences of coefficients defined by the generating functions given in \cref{theorem1} satisfy the recursion
\begin{align}
a_{p,q}^{n+1}=&\bdeltab_{p,0}\sum_{k=0}^n{\frac{a_{k,q}^n}{k+1}}-\Theta_{p,1}\frac{a_{p-1,q}^n}{p(p+1)}+\bdeltab_{p,1}\left[\tilde\Theta_{q,n-2}\frac{(-1)^n}{\dfac{n-q-2}\dfac{q}(n-q)^2}\right.\notag\\
	      &\left.+\bdeltab_{q,n-1}\frac{(-1)^{n-1}}{\dfac{n-1}}-\bdeltab_{q,n}\left[\sum_{k=0}^{n-2}{\frac{(-1)^n}{\dfac{k}\dfac{n-k-2}(k+2)^2}}+\frac{(-1)^{n-1}n}{(n-1)!(n+1)!}\right]\right.\notag\\
	      &+\left.\sum_{k=0}^n{\frac{b_{k,q}^n}{k+2}}-\sum_{k=0}^{q-2}{\frac{b_{k,q-k-2}^n}{k+2}}+\sum_{k=0}^{q-2}{\frac{c_{k,q-k-2}^n}{k+2}}\right], \quad n\geq 1.
\end{align}
\end{lemma}
\begin{proof}
Applying \cref{lemmarecursionBeta} and computing the generating functions of both sides of the asserted equality we find that the claim of the lemma is equivalent to
\begin{align}
\label{eq-recursionAqGF}
\frac{1}{z}\operatorname{A}_{p,q}(z)=&\bdeltab_{p,0}\frac{\operatorname{S}^2(z)}{z(1-z)}\operatorname{A}_{1,q}(z)-\Theta_{p,1}\frac{\operatorname{A}_{p-1,q}(z)}{p(p+1)}\\
&+\bdeltab_{p,1}\frac{(-z)^q}{\dfac{q}}\left[\left(\frac{\operatorname{S}^2(z)}{2(1-z)}+\frac{\operatorname{S}^1(z)}{z}\right)\left(\operatorname{G}(z)-\alpha(z)\right)+\frac{\operatorname{H}(z)}{2}+\operatorname{S}^3(z)+\operatorname{U}^1(z)+z+1\right]\notag,
\end{align}
where explicit expressions for
\begin{align*}
\operatorname{S}^3(z) \coloneqq \sum _{n=q+2}^{\infty } \frac{(-z)^{n-q}}{\dfac{n-q-2} (n-q)^2}\quad &\text{ and } \quad \operatorname{U}^1(z) \coloneqq \sum _{k=1}^{\infty } \frac{(-z)^{k+2} }{ k! (k+2)!}\sum _{l=2}^k \frac{2 l+1}{l (l+1)} 
\end{align*}
are derived in \cref{formula-S3,formula-U1}. For $p=0$ and $p>1$ \cref{eq-recursionAqGF} follows immediately from the defining equations \labelcref{eq-DefinitionA0,eq-DefinitionAp}. For $p=1$ the claim follows from combining \cref{lemma-relation1,lemma-relation2}.
\end{proof}

\begin{lemma}
\label{lemmarecursionB}
For every integers $p,q\geq 0$ the sequences of coefficients defined by the generating functions given in \cref{theorem1} satisfy the recursion
\begin{align}
\label{eq-lemmarecursionB}
b_{p,q}^{n+1} =&\bdeltab_{p,1}\sum_{k=0}^n{\frac{a_{k,q}^n}{k+2}}+\bdeltab_{p,0}\left[\bdeltab_{q,n-1}\frac{(-1)^{n-1}}{\dfac{n-1}}+\tilde\Theta_{q.n-2}\frac{(-1)^n}{\dfac{n-q-2}\dfac{q}(n-q-1)^2}\right.\notag\\
	      &\left.+\sum_{k=0}^n{\frac{b_{k,q}^n}{k+1}}\right]-\Theta_{p,1}\left[\frac{b^n_{p-1,q}}{p(p+1)}+\bdeltab_{q,n-p-1}\frac{(-1)^n(2p+1)}{\dfac{p}\dfac{n-p-1}p(p+1)}\right],\quad n\geq 1.
\end{align}
\end{lemma}
\begin{proof}
The same as before. We show the equality for every $n$ by showing the equality of the generating functions. We obtain that the claim is equivalent to
\begin{align}
\label{eq-lemmarecursionBGF}
\frac{\operatorname{B}_{p,q}(z)}{z}=&\bdeltab_{p,1}\left[\frac{\operatorname{S}^1(z)}{z}+\frac{\operatorname{S}^2(z)}{2(1-z)}\right]\operatorname{A}_{1,q}(z)+\bdeltab_{p,0}\frac{(-z)^{q+1}}{\dfac{q}}\Bigg[\operatorname{S}^4(z)+ \operatorname{U}^2(z)+z\notag\\
		     &+\frac{\operatorname{S}^2(z)}{z(1-z)}\left[\operatorname{G}(z)-\alpha(z)\right]+\operatorname{H}(z)\Bigg]-\Theta_{p.1}\left[\frac{\operatorname{B}_{p-1,q}(z)}{p(p+1)}+\frac{(-z)^{p+q+1}(2p+1)}{\dfac{p}\dfac{q}p(p+1)}\right],
\end{align}
where the functions
\begin{align*}
\operatorname{S}^4(z)	\coloneqq\sum _{n=q+2}^{\infty } \frac{(-z)^{n-q}}{\dfac{n-q-2}(n-q-1)^2} \quad \text{ and }\quad \operatorname{U}^2(z)	\coloneqq\sum _{k=1}^{\infty } \frac{(-z)^{k+2} }{(k+1)!^2}\sum _{l=2}^k \frac{2 l+1}{l (l+1)} 
\end{align*}
are evaluated in \cref{formula-S4,formula-U2}. For $p=0$ \cref{eq-lemmarecursionBGF} follows from the observation that
\begin{equation*}
\operatorname{H}(z)=z\left[\operatorname{S}^4(z)+\operatorname{U}^2(z)+z+\operatorname{H}(z)\right].
\end{equation*}
Next we observe that \cref{eq-DefinitionBp} implies
\begin{align*}
\operatorname{B}_{p,q}(z)+\frac{z\operatorname{B}_{p-1,q}(z)}{p(p+1)} =& \frac{(-z)^{p+q+2}}{\dfac{p}\dfac{q}}\sum_{k=2}^p{\frac{2k+1}{k(k+1)}}-\frac{(-z)^{p+q+2}}{\dfac{p}\dfac{q}}\sum_{k=2}^{p-1}{\frac{2k+1}{k(k+1)}}\\
					 =&\frac{(-z)^{p+q+2}}{\dfac{p}\dfac{q}}\frac{2p+1}{p(p+1)},
\end{align*}
and thus \cref{eq-lemmarecursionBGF} also holds for $p>1$. Finally, for $p=1$ we need to show that
\begin{equation*}
\frac{\operatorname{B}_{1,q}(z)}{z}=\left[\frac{\operatorname{S}^1(z)}{z}+\frac{\operatorname{S}^2(z)}{2(1-z)}\right]\operatorname{A}_{1,q}(z)-\frac{\operatorname{B}_{0,q}(z)}{2}-\frac{3(-z)^{q+1}}{4\dfac{q}},
\end{equation*}
which, after using the defining equations \labelcref{eq-DefinitionA,eq-DefinitionB} several times, amounts to showing that
\begin{equation*}
\left[\frac{\operatorname{S}^2(z)}{2(1-z)}+\frac{1}{z}\right]\operatorname{G}(z)+\left[\frac{\operatorname{S}^1(z)}{z}+\frac{\operatorname{S}^2(z)}{1-z}+\frac{1}{z}\right]\alpha(z)+\frac{\operatorname{H}(z)}{2}+\frac{3z^2}{4}=0,
\end{equation*}
which is exactly what \cref{lemma-relation1} asserts.
\end{proof}
\begin{lemma}
\label{lemmarecursionC}
For every integers $p,q\geq 0$ the sequences of coefficients defined by the generating functions given in \cref{theorem1} satisfy the recursion
\begin{align}
\label{eq-lemmarecursionC}
c_{p,q}^{n+1} =& \bdeltab_{p,1}\sum_{k=0}^n{\frac{a_{k,q}^n}{k+2}}-\Theta_{p,1}\frac{c_{p-1,q}^n}{p(p+1)}+\bdeltab_{p,0}\left[\tilde\Theta_{q,n-2}\frac{(-1)^n}{\dfac{n-q-2}\dfac{q}(n-q-1)^2}\right.\notag\\
	      &\left.+\bdeltab_{q,n-1}\left(\sum_{k=0}^{n-2}{\frac{(-1)^n}{\dfac{k}\dfac{n-k-2}(k+1)^2}}-\frac{(-1)^{n-1}}{n!^2}\right)\right.\notag\\
	      &\left.+\sum_{k=0}^n{\frac{b^n_{k,q}}{k+1}}-\sum_{k=0}^{q-1}{\frac{b_{k,q-k-1}^n}{k+1}}+\sum_{k=0}^{q-1}{\frac{c_{k,q-k-1}^n}{k+1}}\right],\quad n\geq 1.
\end{align}
\end{lemma}
\begin{proof}
The proof goes along the same lines as the previous ones. Equating the generating functions of both sides and using \cref{lemma-relation1} we need to show that
\begin{align}
 \label{eq-recursionCGF}
\frac{\operatorname{C}_{p,q}(z)}{z}=&\bdeltab_{p,1}\left[\frac{\operatorname{S}^1(z)}{z}+\frac{\operatorname{S}^2(z)}{2(1-z)}\right]\operatorname{A}_{1,q}(z)-\Theta_{p,1}\frac{\operatorname{C}_{p-1,q}(z)}{p(p+1)}\notag\\
		     &+\bdeltab_{p,0}\left[\frac{(-z)^q}{\dfac{q}}\left(\frac{\operatorname{S}^2(z)}{z(1-z)}(\operatorname{G}(z)-\alpha(z))+\operatorname{H}(z)+\operatorname{S}^4(z)+\operatorname{U}^2(z)+\frac{z}{q+1}\right)\right.\notag\\
		     &+\left.\frac{5z}{4}\sum _{k=1}^{q-1} \frac{2 (-z)^q}{\dfac{q-k-1} (k+1)!^2}-\sum_{k=0}^{q-1}{\frac{(-z)^{q+1}}{\dfac{k}\dfac{q-k-1}(k+1)^2}}\right.\notag\\
		     &\left.-\sum_{k=1}^{q-1}\frac{(-z)^{q+1}}{\dfac{q-k-1}(k+1)!^2}\sum_{l=2}^k{\frac{2l+1}{l(l+1)}}-\frac{z^{q+1}}{2}\sum_{k=1}^{q-1}{\frac{2(-1)^{k-1}}{(k+1)!^2}d_{q-k-1}}\right.\notag\\
		     &\left.+\Theta_{q,1}\left(z^{q+1}d_{q-1}-\frac{(-z)^{q+1}}{\dfac{q-1}}\right)\right].
\end{align}
For $p\geq 1$ \cref{eq-recursionCGF} is immediately clear from the defining equations \labelcref{eq-DefinitionC1,eq-DefinitionCp}. For $p=0$ we show that the $q$-generating functions coincide. Doing this we obtain after some algebra that \cref{eq-lemmarecursionC} is equivalent to
\begin{align}
\label{eq-recursionCqGF}
0=&\left[\frac{z-1}{z}\operatorname{H}(z)+\operatorname{S}^4(z)+z+\operatorname{U}^2(z)-\zeta z^2\right]\operatorname{T}^1(\zeta z)+z\left[1-\frac{\operatorname{S}^2(z)}{2}\right]\notag\\
  &+\left[\zeta z^2\left(1-\frac{\operatorname{S}^2(\zeta z)}{2}\right)-z\right]D(\zeta z)+\frac{5z}{4}\operatorname{\Sigma}^3(\zeta z)+z\left[\operatorname{\Sigma}^4(\zeta z)+\operatorname{\Upsilon}^2(\zeta z)\right]
\end{align}
with the functions
\begin{align*}
\operatorname{\Sigma}^3(\zeta z) =&\sum_{q=0}^\infty{\sum _{k=1}^{q-1} \frac{2 (-\zeta z)^q}{\dfac{q-k-1} (k+1)!^2}} ,\\
\operatorname{\Sigma}^4(\zeta z)	  =&\sum _{q=0}^{\infty } \sum _{k=0}^{q-1} \frac{(-\zeta z)^q}{\dfac{k} \dfac{q-k-1} (k+1)^2} \qquad\qquad& \text{ and }\\
\operatorname{\Upsilon}^2(\zeta z)	  =&\sum _{q=0}^{\infty} \sum _{k=1}^{q-1} \frac{(-\zeta z)^q }{\dfac{q-k-1} (k+1)!^2}\sum _{l=2}^k \frac{2 l+1}{l (l+1)} 
\end{align*}
given in \cref{formula-Sigma3,formula-Sigma4,formula-Upsilon2}. Since all functions occurring in \cref{eq-recursionCqGF} are explicitly known the result follows from basic algebra.
\end{proof}
With this work done we can now prove our main theorem.
\begin{proof}[Proof of \cref{theorem1}]
The Chapman-Kolmogorov equation implies the recursion
\begin{equation}
\label{eq-ChapmannKolmogorov}
K^n({r'},{r})=\int_{\R }{K^{}({r'},{s})K^{n-1}({s},{r})\mathbbm{d}{s}},\quad {r},{r'}\in\R ,\quad n>1,
\end{equation}
where $K^{}$ is the one-step transition kernel of $\Deltab$ given in \cref{eq-onestepkernel}
From this we can first prove the asserted symmetry $K^n({r'},{r})=K^n(-{r'},-{r})$ by induction on $n$. For $n=1$ this is clearly true so assuming that it holds for some $n\geq 0$ we conclude that $K^{n+1}({r'},{r})$ is equal to
\begin{align*}
\int_{\R }{K^{}({r'},{s})K^n({s},{r})\mathbbm{d}{s}}=&\int_{\R }{K^{}(-{r'},-{s})K^n(-{s},-{r})\mathbbm{d}{s}}\\
  =& \int_{\R }{K^{}(-{r'},{s})K^n({s},-{r})\mathbbm{d}{s}}=K^{n+1}(-{r'},-{r}).
\end{align*}
In the next step we prove \cref{eq-kernelMtoDelta}, also by induction on $n$. For $n=1$ the claim is true by \cref{initialconditions}. We now assume that \cref{eq-kernelMtoDelta} holds for some $n\geq 1$. It then follows that for ${r}\geq0$:
\begin{align*}
K^{n+1}({r'},{r}) =& \int_{\R }{K^{}({r'},{s})K^n({s},{r})\mathbbm{d}{s}}\\
			=& \int_{-\infty}^0{K^{}({r'},{s})K^n({s},{r})\mathbbm{d}{s}}+\int_0^{r}{K^{}({r'},{s})K^n({s},{r})\mathbbm{d}{s}}+\int_{r}^\infty{K^{}({r'},{s})K^n({s},{r})\mathbbm{d}{s}}\\
			=&\sum_{p,q=0}^na_{p,q}^ne^{-(q+2){r}}\int_{-\infty}^0{K^{}({r'},{s})e^{p{s}}\mathbbm{d}{s}}+\frac{(-1)^{n-1}e^{-(n+1){r}}}{\dfac{n-1}}\int_0^{r}{K^{}({r'},{s})e^{s}\mathbbm{d}{s}}\\
			 &+\sum_{p=0}^{n-1}{\frac{(-1)^ne^{-(n-p){r}}}{\dfac{p}\dfac{n-p-2}}}\int_0^{r}{K^{}({r'},{s}){s} e^{-p{s}} \mathbbm{d}{s}}\\
			 &+\sum_{p,q=0}^n {b_{p,q}^n e^{-(q+2) {r}}\int_0^{r}{K^{}({r'},{s}) e^{-p{s}}}            \mathbbm{d}{s}}+\frac{(-1)^{n-1}e^{-{r}}}{\dfac{n-1}}\int_{r}^\infty{K^{}({r'},{s})e^{-(n-1){s}}\mathbbm{d}{s}}\\
			 &+\sum_{p=0}^{n-1}{\frac{(-1)^ne^{-(n-p){r}}{r}}{\dfac{p}\dfac{n-p-2}}}\int_{r}^\infty{K^{}({r'},{s}) e^{-p{s}} \mathbbm{d}{s}}\\
			 &+\sum_{p,q=0}^n {c_{p,q}^n e^{-(q+2) {r}}\int_{r}^\infty							{K^{}({r'},{s}) e^{-p{s}}}            \mathbbm{d}{s}}.
\end{align*}
The five types of integrals occurring in this expression are easily evaluated to give for $p\geq 0$:
\begin{align*}
\int_{-\infty}^0{K^{}({r'},{s})e^{p{s}}\mathbbm{d}{s}}=&\begin{cases}
                                                                    \frac{1}{p+1}-\frac{e^{(p+1){r'}}}{(p+1)(p+2)}	& {r'}\leq 0\\
								    \frac{e^{-{r'}}}{p+2}				&{r'}>0
                                                                   \end{cases},\\
\int_0^{r}{K^{}({r'},{s})e^{{s}}\mathbbm{d}{s}}=&\begin{cases}
                                                                e^{{r'}}-e^{{r'}-{r}}	& {r'}\leq 0\\
								1-e^{{r'}-{r}}+{r'}	& 0<{r'}\leq{r}\\
								{r} 				& {r'}>{r}
                                                               \end{cases},\\
\int_0^{r}{K^{}({r'},{s}){s} e^{-p{s}}\mathbbm{d}{s}}=&\begin{cases}
                                                                         -\frac{e^{{r'}-(p+2){r}}{r}}{p+2}+\frac{e^{{r'}}-e^{{r'}-(p+2){r}}}{(p+2)^2}	& {r'}\leq 0\\
									 \frac{1}{(p+1)^2}-\frac{e^{-(p+1){r'}}(2p+3)}{(p+1)^2(p+2)^2}-\frac{e^{-(p+1){r'}}{r'}}{(p+1)(p+2)}-\frac{e^{{r'}-(p+2){r}}}{(p+2)^2}-\frac{e^{{r'}-(p+2){r}}{r}}{p+2}   & 0<{r'}\leq{r}\\
									 -\frac{e^{-(p+1){r}}{r}}{p+1}+\frac{1-e^{-(p+1){r}}}{(p+1)^2}   & {r'}>{r}
                                                                        \end{cases},\\
\int_0^{r}{K^{}({r'},{s})e^{-p{s}}\mathbbm{d}{s}}=&\begin{cases}
                                                                  \frac{e^{{r'}}-e^{{r'}-(p+2){r}}}{p+2}& {r'}\leq 0\\
								 \frac{1}{p+1}-\frac{e^{{r'}-(p+2){r}}}{p+2}-\frac{e^{-(p+1){r'}}}{(p+1)(p+2)} & 0< {r'}\leq{r}\\
								  \frac{1-e^{-(p+1){r}}}{p+1} & {r'}>{r}
                                                                 \end{cases},\\
\int_{r}^\infty{K^{}({r'},{s})e^{-p{s}}\mathbbm{d}{s}}=&\begin{cases}
                                                                       \frac{e^{{r'}-(p+2){r}}}{p+2}	& {r'}\leq{r}\\
									\frac{e^{-(p+1){r}}}{p+1}-\frac{e^{-(p+1){r'}}}{(p+1)(p+2)}	&{r'}>{r}
                                                                      \end{cases}.
\end{align*}
This implies that for ${r'}\leq 0$ the function $K^{n+1}$ is given by
\begin{equation}
\label{eq-KernelRecursionA}
K^{n+1}({r'},{r})=\sum_{p,q=0}^n {\tilde a_{p,q}^{n+1}e^{p{r'}-(q+2) {r}}},
\end{equation}
where
\begin{align}
\label{eq-tildea-rec}
\tilde a_{p,q}^{n+1}=&\bdeltab_{p,0}\sum_{k=0}^n{\frac{a_{k,q}^n}{k+1}}-\Theta_{p,1}\frac{a_{p-1,q}^n}{p(p+1)}+\bdeltab_{p,1}\left[\bdeltab_{q,n-1}\frac{(-1)^{n-1}}{\dfac{n-1}}-\bdeltab_{q,n}\frac{(-1)^{n-1}}{\dfac{n-1}}\right.\notag\\
	      &\left.+\tilde\Theta_{q,n-2}\frac{(-1)^n}{\dfac{n-q-2}\dfac{q}(n-q)^2}\right.\notag\\
	      &\left.-\bdeltab_{q,n}\sum_{k=0}^{n-2}{\frac{(-1)^n}{\dfac{k}\dfac{n-k-2}(p+2)^2}}\right.\notag\\
	      &+\left.\sum_{k=0}^n{\frac{b_{k,q}^n}{k+2}}-\sum_{k=0}^{q-2}{\frac{b_{k,q-k-2}^n}{k+2}}+\sum_{k=0}^{q-2}{\frac{c_{k,q-k-2}^n}{k+2}}+\bdeltab_{q,n}\frac{(-1)^{n-1}}{\dfac{n-1}(n+1)}\right],
\end{align}
By \cref{lemmarecursionA}, $\tilde a^{n+1}_{p,q}$ is equal to $a_{p,q}^{n+1}$. Similarly, for $0<{r'}\leq{r}$ the function $K^{n+1}$ takes the form
\begin{equation}
\label{eq-KernelRecursionB}
K^{n+1}({r'},{r}) = \sum_{q=0}^n{\beta^{n+1}_q e^{{r'}-(q+2){r}}}+\sum_{p=0}^{n-1}{\frac{(-1)^{n+1}{r'}e^{-p({r'}-{r})-(n+1){r}}}{\dfac{p}\dfac{n-p-1}}}+\sum_{p=0}^{n+1}\sum_{q=0}^{n+1} {\tilde b_{p,q}^{n+1} e^{-p{r'}-(q+2) {r}}}, 
\end{equation}
where
\begin{align*}
\beta^{n+1}_q =	& \bdeltab_{q,n}\left[\frac{(-1)^{n-1}}{\dfac{n-1}(n+1)}-\frac{(-1)^{n-1}}{\dfac{n-1}}-\sum_{k=0}^{n-2}{\frac{(-1)^n}{\dfac{k}\dfac{n-k-2}(k+2)^2}}\right]\\
		&-\sum_{k=0}^{q-2}{\frac{b^n_{k,q-k-2}}{k+2}}+\sum_{k=0}^{q-2}{\frac{c^n_{k,q-k-2}}{k+2}}
\end{align*}
and
\begin{align} 
\label{eq-tildeb-rec}
\tilde b_{p,q}^{n+1} =&\bdeltab_{p,1}\sum_{k=0}^n{\frac{a_{k,q}^n}{k+2}}-\Theta_{p,1}\bdeltab_{q,n-p-1}\frac{(-1)^n(2p+1)}{\dfac{p-1}\dfac{n-p-1}p^2(p+1)^2}\notag\\
	      &+\bdeltab_{p,0}\left[\bdeltab_{q,n-1}\frac{(-1)^{n-1}}{\dfac{n-1}}+\tilde\Theta_{q,n-2}\frac{(-1)^n}{\dfac{n-q-2}\dfac{q}(n-q-1)^2}+\sum_{k=0}^n{\frac{b_{k,q}^n}{k+1}}\right].
\end{align}
\cref{lemmarecursionBeta} implies that $\beta^{n+1}_q=\bdeltab_{q,n}\frac{(-1)^n}{\dfac{n}}$ and by \cref{lemmarecursionB}, $\tilde b_{p,q}^{n+1}$ is equal to $b_{p,q}^{n+1}$. Finally, for ${r'}>{r}$ the function $K^{n+1}$ becomes
\begin{equation}
\label{eq-KernelRecursionC}
K^{n+1}({r'},{r})=\frac{(-1)^ne^{-n{r'}-{r}}}{\dfac{n}}+\sum_{p=0}^{n-1}{\frac{(-1)^{n+1}{r} e^{-p({r'}-{r})-(n+1){r}}}{\dfac{p}\dfac{n-p-1}}}+\sum_{p=0}^{n+1}\sum_{q=0}^{n+1} {\tilde c_{p,q}^{n+1} e^{-p{r'}-(q+2) {r}}},
\end{equation}
where
\begin{align}
\label{eq-tildec-rec}
\tilde c_{p,q}^{n+1} =& \bdeltab_{p,1}\sum_{k=0}^n{\frac{a_{k,q}^n}{k+2}}-\Theta_{p,1}\frac{c_{p-1,q}^n}{p(p+1)}+\bdeltab_{p,0}\left[\tilde\Theta_{q,n-2}\frac{(-1)^n}{\dfac{n-q-2}\dfac{q}(n-q-1)^2}\right.\notag\\
	      &\left.+\bdeltab_{q,n-1}\sum_{k=0}^{n-2}{\frac{(-1)^n}{\dfac{k}\dfac{n-k-2}(k+1)^2}}+\sum_{k=0}^n{\frac{b_{k,q}^n}{k+1}}-\sum_{k=0}^{q-1}{\frac{b_{k,q-k-1}^n}{k+1}}\right.\notag\\
	      &\left.-\bdeltab_{q.n-1}\frac{(-1)^{n-1}}{n!^2}+\sum_{k=0}^{q-1}{\frac{c_{k,q-k-1}^n}{k+1}}\right].
\end{align}
In \cref{lemmarecursionC} it was shown that $\tilde c_{p,q}^{n+1}$ equals $c_{p,q}^{n+1}$. Combining \cref{eq-KernelRecursionA,eq-KernelRecursionB,eq-KernelRecursionC} proves the theorem because it follows that for ${r}\geq0$ the values $K^{n+1}({r'},{r})$ are given by
\begin{equation*}
\begin{cases}
                \sum_{p,q=0}^n {a_{p,q}^{n+1}e^{p{r'}-(q+2) {r}}} 			& {r'}\leq0 	\\
		\frac{(-1)^ne^{{r'}-(n+2){r}}}{\dfac{n}}+\sum_{p=0}^{n-1}{\frac{(-1)^{n+1}{r'}e^{-p({r'}-{r})-(n+1){r}}}{\dfac{p}\dfac{n-p-1}}}+\sum_{p=0}^{n+1}\sum_{q=0}^{n+1} {b_{p,q}^{n+1} e^{-p{r'}-(q+2) {r}}} 	& 0<{r'}\leq{r} 	\\
		\frac{(-1)^ne^{-n{r'}-{r}}}{\dfac{n}}+\sum_{p=0}^{n-1}{\frac{(-1)^{n+1}{r} e^{-p({r'}-{r})-(n+1){r}}}{\dfac{p}\dfac{n-p-1}}}+\sum_{p=0}^{n+1}\sum_{q=0}^{n+1} {c_{p,q}^{n+1} e^{-p{r'}-(q+2) {r}}}		& {r'}>{r}
                \end{cases}.
\end{equation*}
\end{proof}

\section{Discussion}
The way in which \cref{theorem1} is proved gives little insight into how one arrives at the expressions \labelcref{eq-DefinitionA,eq-DefinitionB,eq-DefinitionC} for the generating functions $\operatorname{A}_{p,q}$, $\operatorname{B}_{p,q}$ and $\operatorname{C}_{p,q}$ in the first place. It appears to be indicated to briefly comment on how we derived these formulas. The first step was to compute the kernels $K^n$ for low values of $n$ from the Chapman-Kolmogorov equation \labelcref{eq-ChapmannKolmogorov} and to observe that they have the form asserted in \cref{theorem1}. In the next step we guessed the expression for the part of $K^n(r',r)$ not involving the coefficients $a^n_{p,q}$, $b^n_{p,q}$ and $c^n_{p,q}$ so that the problem was reduced to solving the recurrence equations \labelcref{eq-tildea-rec,eq-tildeb-rec,eq-tildec-rec}. Assuming the validity of \cref{lemmarecursionBeta} it turns out that the first two of these recurrence equation can relatively easily be solved first for $\operatorname{G}(z)$, which is --- up to the factor $(-z)^q/\dfac{q}$ --- the generating function of $(a^n_{1,q}+b^n_{1,q})_{n\geq 1}$, and then also for $\operatorname{A}_{p,q}$ and then $\operatorname{B}_{p,q}$. The third recursion for $(c^n_{p,q})$ was simplified by the empirical observation that 
\begin{equation*}
\sum_{k=0}^{q-1}\frac{b^n_{k,q-k-1}}{k+1}-\sum_{k=0}^{q-1}\frac{c^n_{k,q-k-1}}{k+1}=\bdeltab_{q,n+1}\left[\frac{(-1)^{n-1}}{n!^2}+\sum_{k=0}^{q-1}{\frac{(-1)^q}{\dfac{k}\dfac{q-k-1}(k+1)^2}}-d_q\right]
\end{equation*}
for some real numbers $d_q$, $q\geq 0$, and then solved for $\operatorname{C}_{p,q}$. The educated guesses made in the course of this derivation are justified {\it ex posteriori} by the proofs presented in this paper.

Our original motivation was to derive an explicit expression for the asymptotic variance as in \cref{eq-formulasigma}. For this purpose, knowledge of the generating function of the coefficients of the $n$-step transition kernel, as opposed to knowledge of the coefficients themselves, is sufficient. In order to evaluate the infinite sum appearing in \cref{eq-formulasigma} one is primarily interested in sums of the form $\sum_{n\geq 1}{a^n_{p,q}}$ which is equal to	 $\operatorname{A}_{p,q}(1)$, provided this number is finite. Carrying out the computations, however, turns out to be quite subtle and the results will be reported elsewhere.

It is a natural question if the results presented in this paper can be extended to the first\hyp{}passage percolation problem on $\NN\times\{0,1,\ldots,k\}$, $k\geq 2$. Conceptually, our approach carries over to this setting only if one considers semi-directed percolation in which the horizontal edges may be traversed in only one direction; the combinatorics involved in computing the one-step transition kernel of the Markov chain $\Deltab$ as well as the explicit iteration of the Chapman-Kolmogorov equation \labelcref{eq-ChapmannKolmogorov}, however, soon become unmanageable for larger values of $k$. For the undirected first\hyp{}passage percolation problem there is the possibility that the shortest path $\{(0,0)=p_0,p_1,\ldots,p_{N-1},p_N=(n,0)\}$, $p_i=(x_i,y_i)$, between $(0,0)$ and $(n,0)$ {\it backtraces} by which we mean that there exist indices $0\leq i<j\leq N$ such that $x_j<x_i$. The possible occurrence of such configurations prevents an extension of our recursive method to broader graphs in the undirected setting. One might also wonder if similar results can be obtained for more general class of edge-weight distributions $\Pb$. It is easy to see that the Markov property of $\Deltab$ does not depend on the choice of $\Pb$ and an analysis of our proofs shows that the validity of the central limit theorem as well as expression \labelcref{eq-formulasigma} for the asymptotic variance is not affected by choosing a different edge-weight distribution either, provided one can prove that the stationary distribution $\tilde\pi$ and the one-step kernel $K$ satisfy the moment and mixing conditions used in the proof of \cref{theoremCLT}. It is however, very difficult, to evaluate the formula for the $n$-step transition kernel explicitly, if $\Pb$ is not the exponential distribution, although our approach via generating functions remains likewise applicable.

\paragraph{Acknowledgements}
The work on this paper was completed during a stay at the statistics department of Colorado State University whose hospitality the author gratefully acknowledges. The author also gratefully acknowledges financial support from Technische Universit\"at M\"unchen - Institute for Advanced Study funded by the German Excellence Initiative and from the International Graduate School of Science and Engineering.


\begin{thebibliography}{21}
\providecommand{\natexlab}[1]{#1}
\providecommand{\url}[1]{\texttt{#1}}
\expandafter\ifx\csname urlstyle\endcsname\relax
  \providecommand{\doi}[1]{doi: #1}\else
  \providecommand{\doi}{doi: \begingroup \urlstyle{rm}\Url}\fi

\bibitem[Abramowitz and Stegun(1992)]{abramowitz1992}
M.~Abramowitz and I.~A. Stegun, editors.
\newblock \emph{Handbook of mathematical functions with formulas, graphs, and
  mathematical tables}.
\newblock Dover Publications Inc., New York, 1992.
\newblock ISBN 0-486-61272-4.
\newblock Reprint of the 1972 edition.

\bibitem[Ahlberg(2009)]{ahlberg2009}
D.~Ahlberg.
\newblock Asymptotics of first\hyp{}passage percolation on 1-dimensional
  graphs.
\newblock \emph{Preprint - Department of Mathematical Sciences, Chalmers
  University of Technology and G\"oteborg University}, 39, 2009.

\bibitem[Aldous et~al.(1997)Aldous, Lov{\'a}sz, and Winkler]{aldous1997}
D.~Aldous, L.~Lov{\'a}sz, and P.~Winkler.
\newblock Mixing times for uniformly ergodic {M}arkov chains.
\newblock \emph{Stoch. Process. Their Appl.}, 71\penalty0 (2):\penalty0
  165--185, 1997.

\bibitem[Altmann(1993)]{altmann1993}
M.~Altmann.
\newblock Reinterpreting network measures for models of disease transmission.
\newblock \emph{Social Networks}, 15\penalty0 (1):\penalty0 1--17, 1993.
\newblock ISSN 0378-8733.

\bibitem[Bhamidi et~al.(2010)Bhamidi, van~der Hofstad, and
  Hooghiemstra]{bhamidi2010}
S.~Bhamidi, R.~van~der Hofstad, and G.~Hooghiemstra.
\newblock First passage percolation on random graphs with finite mean degrees.
\newblock \emph{Ann. Appl. Probab.}, 20\penalty0 (5):\penalty0 1907--1965,
  2010.

\bibitem[Chatterjee and Dey(2009)]{chatterjee2009arxiv}
S.~Chatterjee and P.~S. Dey.
\newblock Central limit theorem for first\hyp{}passage percolation time across
  thin cylinders.
\newblock 2009.
\newblock Preprint: available at
  \href{http://arxiv.org/abs/0911.5702v2}{arXiv:0911.5702v2}.

\bibitem[Chen(1999)]{chen1999}
X.~Chen.
\newblock Limit theorems for functionals of ergodic {M}arkov chains with
  general state space.
\newblock \emph{Mem. Am. Math. Soc.}, 139\penalty0 (664):\penalty0 xiv+203,
  1999.
\newblock ISSN 0065-9266.

\bibitem[Flaxman et~al.(2011)Flaxman, Gamarnik, and Sorkin]{flaxman2006fpp}
A.~Flaxman, D.~Gamarnik, and G.~B. Sorkin.
\newblock First-passage percolation on a ladder graph, and the path cost in a
  {VCG} auction.
\newblock \emph{Random Structures Algorithms}, 38\penalty0 (3):\penalty0
  350--364, 2011.
\newblock ISSN 1042-9832.

\bibitem[Graham et~al.(1995)Graham, Gr{\"o}tschel, and Lov{\'a}sz]{graham1995}
R.~L. Graham, M.~Gr{\"o}tschel, and L.~Lov{\'a}sz, editors.
\newblock \emph{Handbook of combinatorics. {V}ol. 2}.
\newblock Elsevier Science B.V., Amsterdam, 1995.

\bibitem[Grimmett and Kesten(1984)]{grimmet1984}
G.~Grimmett and H.~Kesten.
\newblock First\hyp{}passage percolation, network flows and electrical
  resistances.
\newblock \emph{Z. Wahrsch. Verw. Gebiete}, 66\penalty0 (3):\penalty0 335--366,
  1984.
\newblock ISSN 0044-3719.

\bibitem[Hammersley and Welsh(1965)]{hammersley1965}
J.~M. Hammersley and D.~J.~A. Welsh.
\newblock First\hyp{}passage percolation, subadditive processes, stochastic
  networks, and generalized renewal theory.
\newblock In \emph{Proc. {I}nternat. {R}es. {S}emin., {S}tatist. {L}ab.,
  {U}niv. {C}alifornia, {B}erkeley, {C}alif}, pages 61--110. Springer, New
  York, 1965.

\bibitem[Kesten(1987)]{kesten1987}
H.~Kesten.
\newblock Percolation theory and first\hyp{}passage percolation.
\newblock \emph{Ann. Probab.}, 15\penalty0 (4):\penalty0 1231--1271, 1987.

\bibitem[Nummelin and Tuominen(1982)]{Nummelin1982}
E.~Nummelin and P.~Tuominen.
\newblock Geometric ergodicity of {H}arris recurrent {M}arkov chains with
  applications to renewal theory.
\newblock \emph{Stoch. Process. Their Appl.}, 12\penalty0 (2):\penalty0
  187--202, 1982.

\bibitem[Renlund(2010)]{renlund2010}
H.~Renlund.
\newblock First-passage percolation with exponential times on a ladder.
\newblock \emph{Comb. Probab. Comput.}, 19\penalty0 (4):\penalty0 593--601,
  2010.
\newblock ISSN 0963-5483.

\bibitem[Schlemm(2009)]{schlemm2009}
E.~Schlemm.
\newblock First\hyp{}passage percolation rates on width\hyp{}two stretches with
  exponential link weights.
\newblock \emph{Electron. Commun. Probab.}, 140:\penalty0 424--434, 2009.

\bibitem[Sepp{\"a}l{\"a}inen(1998)]{seppaelaiinen1998}
T.~Sepp{\"a}l{\"a}inen.
\newblock Exact limiting shape for a simplified model of first\hyp{}passage
  percolation on the plane.
\newblock \emph{Ann. Probab.}, 26\penalty0 (3):\penalty0 1232--1250, 1998.

\bibitem[Smythe and Wierman(1978)]{smythe1978}
R.~T. Smythe and John~C. Wierman.
\newblock \emph{First-passage percolation on the square lattice}, volume 671 of
  \emph{Lecture Notes in Mathematics}.
\newblock Springer, Berlin, 1978.

\bibitem[Sood et~al.(2005)Sood, Redner, and ben Avraham]{sood2005}
V.~Sood, S.~Redner, and D.~ben Avraham.
\newblock First\hyp{}passage properties of the {E}rd{\H o}s-{R}enyi random
  graph.
\newblock \emph{J. Phys. A}, 38\penalty0 (1):\penalty0 109--123, 2005.
\newblock ISSN 0305-4470.

\bibitem[van~der Hofstad et~al.(2001)van~der Hofstad, Hooghiemstra, and
  Van~Mieghem]{hofstad2001}
R.~van~der Hofstad, G.~Hooghiemstra, and P.~Van~Mieghem.
\newblock First\hyp{}passage percolation on the random graph.
\newblock \emph{Probab. Engrg. Inform. Sci.}, 15\penalty0 (2):\penalty0
  225--237, 2001.
\newblock ISSN 0269-9648.

\bibitem[Wilf(2006)]{wilf2006}
H.~S. Wilf.
\newblock \emph{generatingfunctionology}.
\newblock A K Peters Ltd., Wellesley, third edition, 2006.

\bibitem[{Wolfram Research, Inc.}(2010)]{wolframfunctionsite}
{Wolfram Research, Inc.}
\newblock The wolfram functions site.
\newblock \url{http://functions.wolfram.com/}, 2010.

\end{thebibliography}
\end{document}